\documentclass{amsart}

\usepackage{geometry}
\usepackage{graphicx}
\usepackage{url}
\usepackage{amsthm}
\usepackage{amsmath}
\usepackage{amsfonts}
\usepackage{amssymb}
\usepackage{MnSymbol}
\usepackage{enumerate}
\usepackage{extarrows}
\usepackage{comment}
\usepackage{booktabs}
\usepackage{graphicx}
\usepackage{verbatim}
\usepackage{scalerel} %
\usepackage{nicematrix}
\usepackage{mathrsfs}

\usepackage{color}
\definecolor{darkgreen}{rgb}{0,0.5,0}
\usepackage[colorlinks, citecolor=darkgreen, backref]{hyperref}

\definecolor{darkred}{rgb}{0.5,0,0}

\theoremstyle{plain}
\newtheorem{theorem}{Theorem}
\numberwithin{theorem}{section}
\newtheorem{lemma}[theorem]{Lemma}

\newtheorem{cor}[theorem]{Corollary}
\newtheorem{prop}[theorem]{Proposition}

\newtheorem*{lemma*}{Lemma}
\theoremstyle{definition}

\newtheorem{example}[theorem]{Example}

\newcommand{\C}{\mathbb{C}}

\newcommand{\Z}{\mathbb{Z}}
\newcommand{\Q}{\mathbb{Q}}

\newcommand{\co}[1][K]{\mathcal{O}_{#1}}
\newcommand{\eps}{\varepsilon}

\newcommand{\Aut}{\operatorname{Aut}}
\newcommand{\GL}{\operatorname{GL}}
\newcommand{\Tr}{\operatorname{Tr}}
\newcommand{\Nm}{\operatorname{Nm}}
\newcommand{\cl}{\operatorname{cl}}
\newcommand{\fd}{\mathfrak{d}}
\newcommand{\fg}{\mathfrak{g}}
\newcommand{\fn}{\mathfrak{n}}
\newcommand{\fp}{\mathfrak{p}}
\newcommand{\fs}{\mathfrak{s}}
\newcommand{\fw}{\mathfrak{w}}
\newcommand{\gen}{\operatorname{gen}}
\newcommand{\rank}{\operatorname{rank}}
\newcommand{\ord}{\operatorname{ord}}
\newcommand{\Res}{\operatorname{Res}}
\newcommand{\floor}[1]{\left\lfloor #1 \right\rfloor}
\newcommand{\ceil}[1]{\left\lceil #1 \right\rceil}

\DeclareMathOperator*{\bigperp}{\scalerel*{\mathord{\perp}}{\sum}}

\newcommand{\MRU}[2]{U_{\mathcal{O}_{#1}}({#2})}
\newcommand{\UL}[2]{b_{\mathcal{O}_{#1}}({#2})}
\newcommand{\IUL}[2]{a_{\mathcal{O}_{#1}}({#2})}
\newcommand{\SUL}[2]{X_{\mathcal{O}_{#1}}({#2})}
\newcommand{\CS}[2]{\mathcal{S}_{#1}({#2})}
\newcommand{\mass}[2]{\mathrm{mass}_{{#1}}({#2})}
\newcommand{\massI}[2]{\mathrm{mass}_{{#1}}(I_{#2})}
\newcommand{\card}[1]{\# {#1}}
\newcommand{\norm}[1]{\Nm({#1})}
\newcommand{\bigO}[1]{O({#1})}
\newcommand{\nquadcoeff}{\frac{2 \log \Delta_K - 2d \log(2 \pi) - 3d}{8}}
\newcommand{\nlinearcoeffa}{\frac{d \log{(8 \pi )} + d - \log \Delta_K}{4}}
\newcommand{\nlinearcoeff}{ \nlinearcoeffa - \sum_{\fp | 2} \floor{\frac{e_\fp}{2}} \log \norm{\fp} }
\newcommand{\Autn}{\mathcal{A}_K}
\newcommand{\numquadexceptions}{559}

\allowdisplaybreaks

\begin{document}

\title[]{Asymptotics of $n$-universal lattices over number fields}

\begin{abstract}
We prove an explicit asymptotic formula for the logarithm of the minimal ranks of $n$-universal lattices over the ring of integers of totally real number fields. We also show that, for any constant $C>0$ and $n\geq 3$, there are only finitely many totally real fields with an $n$-universal lattice of rank at most $C$, with all such fields being effectively computable. Similarly, for any $n \geq 3$, we show that there are only finitely many totally real fields admitting an $n$-universal criterion set of size at most $C$, with all such fields likewise being effectively computable.

\end{abstract}

\author{Dayoon Park}
\address{Charles University \\ Faculty of Mathematics and Physics \\ Department of Algebra \\ Sokolovsk\'{a} 83 \\ 186 75 Praha 8 \\ Czech Republic}
\email{pdy1016@snu.ac.kr}

\author{Robin Visser}
\address{Charles University \\ Faculty of Mathematics and Physics \\ Department of Algebra \\ Sokolovsk\'{a} 83 \\ 186 75 Praha 8 \\ Czech Republic}
\email{robin.visser@matfyz.cuni.cz}

\author{Pavlo Yatsyna}
\address{Charles University \\ Faculty of Mathematics and Physics \\ Department of Algebra \\ Sokolovsk\'{a} 83 \\ 186 75 Praha 8 \\ Czech Republic}
\email{p.yatsyna@matfyz.cuni.cz}

\author{Jongheun Yoon}
\address{Charles University \\ Faculty of Mathematics and Physics \\ Department of Algebra \\ Sokolovsk\'{a} 83 \\ 186 75 Praha 8 \\ Czech Republic}
\email{jongheun.yoon@matfyz.cuni.cz}

\date{\today}
\thanks{R. V. and P. Y. were supported by {Charles University} programme PRIMUS/24/SCI/010. D. P. and J. Y. were supported by grant 21-00420M from Czech Science Foundation (GA{\v C}R). D. P., R. V. and P. Y. were supported by {Charles University} programme UNCE/24/SCI/022.}

\maketitle

\section{Introduction}

The theory of universal quadratic forms, and more generally universal lattices, has a long history and has been extensively studied by many authors.  In 1770, Lagrange famously showed that every positive integer can be written as a sum of four integer squares.  In the modern terminology, we therefore say that the rank four quadratic form $X^2 + Y^2 + Z^2 + W^2$ is a \emph{universal quadratic form} over $\Z$.  Moreover, as there exists no quadratic form in three variables over $\Z$ which represents all positive integers (e.g. see \cite[Theorem~13, p.~291]{Albert}), this implies that the \emph{minimal rank} of a universal quadratic form over $\Z$ is four.

In general, for any number field $K$ with ring of integers $\co$, determining the minimal rank of a universal quadratic form (or a universal lattice) over $\co$ is a particularly hard problem.  In 1941, Maa{\ss} \cite{Maass1941} proved that the ternary quadratic form $X^2 + Y^2 + Z^2$ is universal over $\co[\Q(\sqrt{5})]$, and subsequently over the last few decades, many authors have proven various different bounds on the minimal ranks of universal quadratic forms and universal lattices over $\co$, e.g. see Chan--Kim--Raghavan \cite{CKR1996}, Earnest--Khosravani \cite{EarnestKhosravani1997}, Kim \cite{Kim1999}, as well as the more recent papers of Blomer and Kala \cite{BlomerKala}, Kim, Kim and Park \cite{KKP2022}, Kala, Yatsyna and {\.Z}mija \cite{KYZ}, and Man \cite{Man24}.  A famous open conjecture of Kitaoka from the 1990s asserts that there are only finitely many totally real number fields that admit a ternary universal integral quadratic form (e.g. see \cite[p.~263]{CKR1996} or \cite{Kim97,KalaYatsyna2023,KKPYZ2025,KKK25} for progress on this problem).

\bigskip
In this paper, we study the more general notion of \emph{$n$-universality}. For a fixed totally real number field $K$, we denote by $\MRU{K}{n}$ the \emph{minimal rank of an $n$-universal $\co$-lattice}, where precise definitions will be given in the next section.  When $K = \Q$, it is well-known that $I_{n+3}$ is $n$-universal and thus $U_\Z(n) = n+3$ for all positive $n \leq 5$  (see \cite{Mordell1930, Ko1937}).  Oh \cite{Oh00} computed the values $U_\Z(n) = 13, 15, 16, 28, 30$ for $n = 6, \dots, 10$, respectively. It is folklore that $U_\Z(n)$ has an easy upper bound $(n+3)\cl_\Q(I_{n+3})$, for $I_{n+3}$ is locally $n$-universal. Oh \cite{Oh07} provided a lower bound for the minimal rank of $n$-regular lattices, which also serves as a lower bound for $U_\Z(n)$ when $n$ is large, since he proved in the same paper that every $n$-regular $\Z$-lattice is $n$-universal for $n \ge 28$.

It is known that $\MRU{K}{n}$ is well-defined for any totally real number field $K$ and for all $n \geq 1$ (e.g. see \cite{ChanOh2023}). %
As mentioned above, numerous authors have studied the values of $\MRU{K}{1}$ for various totally real fields, e.g. see \cite{Earnest98, Kalasurvey, Kim04} and the references therein for further details. In contrast, little is known about $\MRU{K}{n}$ for $K \ne \Q$ and $n \ge 2$.  In 2006, Sasaki \cite{Sasaki2006} showed that $\MRU{K}{2}$ equals $6$ when $K = \Q(\sqrt2)$ or $\Q(\sqrt5)$, and is at least $7$ when $K$ is any other real quadratic field.  Recently, Tinkov\'{a} and the third author have also shown various bounds for $\MRU{K}{n}$ \cite{TinkovaYatsyna}.

\bigskip

Our first theorem is to provide an explicit asymptotic formula for the logarithm of $\MRU{K}{n}$ as $n \to \infty$, in particular explicitly giving the main term and the second-order term for $\log \MRU{K}{n}$:
\begin{theorem} \label{thm:mainasymp}
Let $K$ be a totally real degree $d$ number field of discriminant $\Delta_K$.  Then, with the exception of a finite set of at most $\numquadexceptions$ real quadratic fields $K$, we have
\begin{equation} \label{eq:main}
    \log \MRU{K}{n} = \frac{d}{4} n^2 \log{n} + n^2 \Big( \nquadcoeff \Big)  +  \bigO{n \log n} %
\end{equation}
as $n \to \infty$.
\end{theorem}

In particular, this proves that, for any $\eps > 0$, we have that $\MRU{K}{n} > n^{(d/4 - \eps)n^2}$ for all sufficiently large $n$, thus showing that $\MRU{K}{n}$ grows very quickly as
$n \to \infty$.  %
We also show that, for \emph{all} totally fields $K$ (including all real quadratic fields), we have
\begin{equation} \label{eq:limsup}
    \limsup_{n \to \infty} \frac{ \log \MRU{K}{n}}{(d/4) n^2 \log n} = 1
\end{equation}
and in particular, we have that $(\ref{eq:main})$ holds for infinitely many positive integers $n$.

We can also prove the following effective density result, which is analogous to the results of Pfeuffer \cite{Pfeuffer} which gives a finiteness statement for the class numbers $\cl_K(I_n)$.  Here, we prove a similar result for the minimal rank $\MRU{K}{n}$ and for sizes of $n$-universal criterion sets $\CS{K}{n}$.

\begin{theorem} \label{thm:finite}
    Let $C > 0$ and fix some positive integer $n \geq 3$.  Then there are only finitely many totally real fields $K$ such that $\MRU{K}{n} < C$.  In particular, all such fields $K$ satisfying $\MRU{K}{n} < C$ are contained within an effectively computable finite set of fields.
    
\end{theorem}

To give an explicit bound, if $n \geq 3$ and $\MRU{K}{n} < C$, then we can show that the degree $d$ and discriminant $\Delta_K$ of such fields $K$ satisfy the following bounds:
\begin{equation} \label{eq:maindbound}
    d \leq 
    \frac{4 n \log(C) + 8 \log(n!) + 254n(n-1) - 4\log(2)}{
    \log ( (2 \cdot 60.1^{n(n-1)/4 }  \prod_{i=1}^n \Gamma(i/2) ) / ( 2^{(n+5)/2} \cdot \pi^{n(n+1)/4} \cdot (2 \zeta(n/2) )^{\delta_{2 \mid n}} ) )}
\end{equation}
and 
\begin{equation} \label{eq:maindiscbound}
    \Delta_K \leq 
    \Big( \frac{C^n (n!)^2 \pi^{dn(n+1)/4} 2^{d(n+5)/2} (2 \zeta(n/2))^{d\delta_{2 \mid n}} }{2 \cdot (2 \prod_{i=1}^n \Gamma(i/2) )^d  }  \Big)^{4 / (n(n-1))}.
\end{equation}
where $\delta_{2 \mid n} = 1$ if $n$ is even, and $0$ otherwise.

We remark that Theorem~\ref{thm:finite} thereby proves a strong version of the $n$-universal analogue of Kitaoka’s conjecture, for all $n \geq 3$.

We can also show an analogous theorem giving lower bounds for $\MRU{K}{2}$.  We recall that a field $K$ has a \emph{normal tower} from $\Q$ if there exist fields $K_1, K_2, \dots, K_m$ such that $\Q = K_1 \subseteq K_2 \subseteq \dots \subseteq K_m = K$ and $K_{i+1} / K_i$ is a relatively normal extension for all $i = 1, \dots, m-1$.

\begin{theorem} \label{thm:finite2}
    Let $C > 0$, and fix some positive integer $d \geq 1$.  Then there are only finitely many degree $d$ totally real fields $K$ such that $\MRU{K}{2} < C$.  In particular, if $d \geq 3$, then all such fields $K$ satisfying $\MRU{K}{2} < C$ are contained within an effectively computable finite set of fields. 
    Furthermore, there are only finitely many such fields $K$ in total (of arbitrary degree) which have a normal tower from $\Q$.  There are only finitely many such fields $K$ in total, if either the generalised Riemann hypothesis or the Artin conjecture (on the analyticity of Artin $L$-functions) is true.
\end{theorem}

To give an explicit bound, if $\MRU{K}{2} < C$ and $d \geq 3$, then we have the upper bound on the discriminant $\Delta_K$:
\begin{equation} \label{eq:maindiscbound2}
    \Delta_K \leq \max \left( \exp \big( (12(d-1))^2 \big), \; \Big (\frac{ 2 C^2 \cdot (\pi \cdot 2^{5/2})^d \cdot d \cdot (2d)!   }{ 0.000181 \cdot ((d-1)/e)^{d-1} } \Big)^{\frac{12d}{5d-12}} \right).
\end{equation}

We present a brief outline of our methods we use to bound $\MRU{K}{n}$:
\begin{itemize}
    \item We first reduce the problem to counting the number of indecomposable unimodular $\co$-lattices of rank $m$ for each $m \leq n+3$. This is shown in Lemma~\ref{bounds}.

    \item By denoting $\IUL{K}{m}$ as the number of rank $m$ indecomposable unimodular $\co$-lattices, and $\UL{K}{m}$ as the number of rank $m$ unimodular $\co$-lattices, we use a combinatorial theorem of Wright \cite{Wright68} (Theorem~\ref{thm:wright})  to show that $\IUL{K}{m} \sim \UL{K}{m}$ as $m \to \infty$, for all totally real fields $K$, except for a finite set of exceptional real quadratic fields.

    \item This thus reduces the problem to obtaining bounds on $\UL{K}{n}$.  To obtain a lower bound, we use Siegel's mass formula to compute an explicit asymptotic formula for $\massI{K}{n}$.  To obtain an upper bound, we use again Siegel's mass formula in addition to obtaining a bound on the orders of finite subgroups of $\GL_n(\co)$, applying some known classical bounds of Schur \cite{Schur} (Theorem~\ref{thm:schur}).
\end{itemize}

It is natural to ask whether we can prove similar bounds for $\MRU{K}{n}$, as given in Theorems~\ref{thm:finite} and \ref{thm:finite2}, in the case $n = 1$.  Unfortunately our methods bounding $\MRU{K}{n}$ rely crucially on bounds for the class number $\cl_K(I_n)$ and in particular the Siegel mass constant $\massI{K}{n}$.  While the results of Pfeuffer \cite{Pfeuffer} provide lower bounds for $\massI{K}{n}$ and $\cl_K(I_n)$ that grow with the discriminant $\Delta_K$ for $n \geq 2$, in the case of $n = 1$, we simply have $\cl_K(I_1) = 1$ and $\massI{K}{1} = 1/2$ for all totally real fields $K$.  This unfortunately yields no useful lower bounds for $\MRU{K}{1}$ (besides the trivial lower bound $\MRU{K}{1} \geq 1$). More importantly, even the statement itself cannot hold in the case of $n=1$ for an arbitrary choice of $C$, as we know that there exist infinitely many real quadratic fields with a universal quadratic form of rank $8$ \cite{Kim00}.

Whilst Theorem~\ref{thm:mainasymp} implies that $\MRU{K}{n}$ grows very quickly with $n$, the $\bigO{n \log n}$ error term means that it does not immediately tell us the size of ``short gaps'' of the form $\MRU{K}{n+k} - \MRU{K}{n}$ for small $k$.  We therefore further show that, at least for $k \geq 4$, such gaps must grow with the discriminant $\Delta_K$ of $K$.

\begin{theorem} \label{thm:shortgap}
    Let $C > 0$, fix $d \geq 1$ and fix some even positive integer $n$.  Then there are only finitely many degree $d$ totally real fields $K$ such that $\MRU{K}{n+4} - \MRU{K}{n} < C$. 
    Furthermore, if $n \equiv 2$ mod 6 or if either the generalised Riemann hypothesis or the Artin conjecture is true, then there are only finitely many totally real fields $K$ (of arbitrary degree) such that $\MRU{K}{n+4} - \MRU{K}{n} < C$.  

\end{theorem}

In particular, if $\MRU{K}{n+4} - \MRU{K}{n} < C$ and $d \geq 3$, then the discriminant $\Delta_K$ satisfies the same bound given in (\ref{eq:maindiscbound2}).
If $n \equiv 2$ (mod 6) 
, then the degree $d$ and discriminant $\Delta_K$ furthermore satisfy the bounds given in (\ref{eq:maindbound}) and (\ref{eq:maindiscbound}) with $n$ replaced with $3$.

We can also prove an analogous result to Theorem~\ref{thm:finite} and Theorem~\ref{thm:finite2} for sizes of criterion sets for $n$-universal lattices.  Recall that an \emph{$n$-universal criterion set} is a set $\mathcal{S}$ of rank $n$ $\mathcal{O}_K$-lattices with the property that any lattice $L$ is $n$-universal if and only if $L$ represents all lattices in $\mathcal{S}$.  We denote an $n$-universal criterion set for $\mathcal{O}_K$ as $\CS{K}{n}$. By a theorem of Chan--Oh \cite{ChanOh2023}, it is known that finite $n$-universal criterion sets $\CS{K}{n}$ always exist for any totally real field $K$ and $n \geq 1$.

\begin{theorem} \label{thm:criterionset}
    Let $C > 0$ and let $n \geq 3$.  Then there are only finitely many totally real fields $K$ which admit an $n$-universal criterion set $\CS{K}{n}$ such that $\card{ \CS{K}{n} } < C$.  In particular, if $\card{ \CS{K}{n} } < C$ then the degree $d$ and discriminant $\Delta_K$ of $K$ satisfy the same upper bounds given in (\ref{eq:maindbound}) and (\ref{eq:maindiscbound}).

    Furthermore, for any fixed $d \geq 1$ there are only finitely many degree $d$ totally real fields $K$ which admit a 2-universal criterion set $\CS{K}{2}$ such that $\card{ \CS{K}{2} } < C$.  In particular, if $d \geq 3$ and $\card{ \CS{K}{2} } < C$, then the discriminant $\Delta_K$ of $K$ satisfies the same upper bound given in (\ref{eq:maindiscbound2}).
    
\end{theorem}

As with the minimal rank $\MRU{K}{n}$, this proves that the sizes of $n$-universal criterion sets $\CS{K}{n}$ for $n \geq 2$ grows very quickly for larger fields $K$.

On the other hand, the rapid growth of $\MRU{K}{n}$ with $n$ is connected not only to the size of $n$-universal criterion sets, but also to their uniqueness.
Although it has recently been shown \cite{KKR2024} that $1$-universal criterion set that is minimal with respect to inclusion is always unique up to isometry, it is worth cautioning that, for some fields $K$ and integers $n$, there might not exist a \emph{unique minimal} such $n$-universal criterion set (e.g. see \cite{KLO2022} or \cite{EKK2013} for examples).
Theorem~\ref{thm:criterion nonunique} shows that the uniqueness actually fails to hold for most cases.
We refer the reader to Kala--Kr\'asensk\'y--Romeo \cite{KKR2024} and the references therein for further background on criterion sets.

\begin{theorem} \label{thm:criterion nonunique}
Let $K$ be a totally real number field. For sufficiently large $n$, there exist infinitely many $n$-universal criterion sets for $\co$ of the smallest cardinality up to isometry.
\end{theorem}

\bigskip
This paper is structured as follows.  In Section~\ref{sec:prelim}, we introduce the relevant notation and terminology used in the paper.  In Section~\ref{sec:localstuff}, we prove some preliminary lemmas on representing rank $n$ $\co$-lattices, particularly in the case of dyadic local fields.  In Section~\ref{sec:indecomposables}, we prove lower and upper bounds for $\MRU{K}{n}$ based on the number of indecomposable $\co$-lattices of rank $\leq n$.  In Section~\ref{sec:mass}, we introduce the Siegel mass constant $\mass{K}{L}$ and apply some results of K\"{o}rner \cite{Korner} and Pfeuffer \cite{Pfeuffer71} to prove necessary bounds for $\mass{K}{L}$.  In Section~\ref{sec:unimodular}, we give bounds for the maximal order of a finite subgroup of $\GL_n(\co)$ and use these results, along with our bounds on the Siegel mass constant, to prove bounds for $\IUL{K}{n}$, the number of rank $n$ indecomposable unimodular $\co$-lattices.  

In Section~\ref{sec:lower}, we apply a criterion of Wright \cite{Wright68} to obtain a lower bound for $\IUL{K}{n}$ and thus a lower bound for $\MRU{K}{n}$, and in Section~\ref{sec:upper}, we also obtain upper bounds for $\UL{K}{n}$ and thus an upper bound for $\MRU{K}{n}$.  Putting these bounds together gives us our proof of Theorem~\ref{thm:mainasymp}.  In Section~\ref{sec:effective_density} we prove Theorems~\ref{thm:finite} and \ref{thm:finite2}.  Section~\ref{sec:shortgaps} gives our bounds on the size of small gaps of the form $\MRU{K}{n+k} - \MRU{K}{n}$, and gives our proof of Theorem~\ref{thm:shortgap}.  Finally, in Section~\ref{sec:criterion}, we give some results on $n$-universal criterion sets and give our proof of Theorems~\ref{thm:criterionset} and \ref{thm:criterion nonunique}.

\section*{Acknowledgements}
We are very grateful to V\'{i}t\v{e}zslav Kala for his many valuable insights and comments.  Our discussions with him provided the primary source of motivation for the main results of this paper.  We also thank Byeong-Kweon Oh for many valuable comments and feedback on an earlier draft of this paper.

\section{Preliminaries} \label{sec:prelim}

Throughout this paper, we adopt the notation and terminology of O'Meara \cite{OMeara}.
For convenience, we record here the additional definitions and slight modifications that will be used in the sequel.
For the remainder of this paper, $K$ is a totally real number field of degree $d = [K : \Q]$ with discriminant $\Delta_K$ and $\co$ is its ring of integers.
For a finite place $\fp$ on $\co$, we denote by $K_\fp$ and $\co[\fp]$ the completion of $K$ and $\co$ at $\fp$, respectively.

Let $R$ stand for one of $K$, $\co$, $K_\fp$, or $\co[\fp]$, where $\fp$ is a finite place on $\co$.
A quadratic $R$-module is a finitely generated $R$-module $X$ equipped with a symmetric bilinear form $B : X \times X \to R$ and the associated quadratic form defined by $Q(x) := B(x, x)$ for all $x\in X$.
For any subset $S$ of $X$, the orthogonal complement of $S$ is denoted by $S^\perp = \{ x\in X \mid B(x, S) = 0 \}$.
If $X^\perp = 0$, then we call $X$ nondegenerate.
For quadratic $R$-modules $(X, B, Q)$ and $(X', B', Q')$, if there exists an $R$-linear map $\sigma$ from $X'$ into $X$ such that $Q'(x) = Q(\sigma(x))$ for all $x\in X'$, then we say that $X'$ is represented by $X$ and we write $\sigma : X' \to X$.
Moreover, if $\sigma$ is bijective, then we say that $X'$ is isometric to $X$ and we write $\sigma : X' \cong X$.
The set of all isometries from $X$ onto itself forms a group $\Aut(X)$.
When $X$ represents a line $R x$ with $Q(x)=\alpha$, we simply say that $X$ represents $\alpha$, denoted $\alpha \to X$.

When $X$ is a free $R$-module, we define the Gram matrix of $X$ with respect to a basis $\{ x_1, \ldots, x_n \}$ for $X$ as a symmetric matrix $G_X:=\begin{pmatrix}
    B(x_i,x_j)
\end{pmatrix} \in \operatorname{Mat}_n(R)$.
The discriminant $dX$ of $X$ is the square class $(\det G_X)(R^\times)^2$, regarded as an element in the quotient monoid $R / (R^\times)^2$.
We conventionally identify a free quadratic $R$-module $X$ with its Gram matrix $G_X$, and write $X \cong G_X$.
When $G_X=\begin{pmatrix}
    B(x_i,x_j)
\end{pmatrix}$ is diagonal, i.e. $B(x_i,x_j)=0$ for any $i \ne j$, we simply write $X \cong \left< Q(x_1), \ldots, Q(x_n) \right>$.
By abuse of notation, we may also write $X \cong \left< Q(x_1)(R^\times)^2, \ldots, Q(x_n)(R^\times)^2 \right>$. For instance, when $X$ is a line, we write $X \cong \left< dX \right>$.
We denote by $I_n$ the $n\times n$ identity matrix.

Let $(F, \co[F])$ stand for either $(K, \co)$ or $(K_\fp, \co[\fp])$ for a finite place $\fp$ on $\co$.
A quadratic $F$-module is called a quadratic $F$-space. A quadratic $\co[F]$-module is called an $\co[F]$-lattice.
Let $L$ be an $\co[F]$-lattice.
The rank of $L$ is defined as the $F$-dimension of the quadratic space $L\otimes F$.
If $L$ is an $\co[F]$-submodule of a quadratic $F$-space $V$, it is called a lattice in $V$.
If in addition $L \otimes F = V$, then it is called a lattice on $V$.
The scale $\fs L$ of $L$ is the set $B(L,L)$, which is an ideal of $\co[F]$. The norm $\fn L$ of $L$ is an ideal of $\co[F]$ generated by $Q(L)$.
An $\co[F]$-lattice $L$ of rank $n$ can be written as $L=\mathfrak{a}_1 v_1 + \dotsb + \mathfrak{a}_n v_n$, where $v_1$, \dots, $v_n$ are vectors in $L\otimes K$ and $\mathfrak{a}_1$, \dots, $\mathfrak{a}_n$ are fractional ideals in $K$.
Then its volume $\mathfrak{v}L$ is defined by the set $\mathfrak{v}L:=\det \begin{pmatrix}
    B(v_i,v_j)
\end{pmatrix} \cdot \mathfrak{a}_1^2 \dotsm \mathfrak{a}_n^2$, which is an ideal of $\co[F]$.
It is clear that $\mathfrak{v}L \subseteq (\fs L)^n$.
When $\mathfrak{v}L = (\fs L)^n$, we say that $L$ is $\fs L$-modular, and especially when $\fs L = \co[F]$, $L$ is called unimodular.

We say that an $\co[F]$-lattice $L$ is a direct sum of sublattices $L_1$, \dots, $L_r$ if it is their direct sum as $\co[F]$-modules.
Suppose $L$ is a direct sum $L = L_1 \oplus \dotsb \oplus L_r$ with $B(L_i, L_j) = 0$ for $1 \le i < j \le r$. We then say that $L$ is the orthogonal sum of the sublattices $L_1$, \dots, $L_r$ or that $L$ has the orthogonal splitting $L = L_1 \mathbin{\perp} \dotsb \mathbin{\perp} L_r$.
We say that a sublattice $L'$ splits $L$, or that $L'$ is a component of $L$, if there exists a sublattice $L''$ of $L$ with $L = L' \mathbin{\perp} L''$. A unimodular sublattice is always a component \cite[Corollary~82:15a]{OMeara}.
For square matrices $A$ and $A'$, the symbol $A \mathbin{\perp} A'$ represents a block diagonal square matrix of the form
\[
 \begin{pNiceMatrix}[hlines,vlines] A & 0 \\ 0 & A' \end{pNiceMatrix}\text.
\]
Hence, for free $\co[F]$-lattices $L$ and $L'$, we have $L \mathbin{\perp} L' \cong G_L \mathbin{\perp} G_{L'}$.

We denote by $K^+$ and $\co^+$ the set of totally positive elements in $K$ and $\co$, respectively, and write $\alpha \succ 0$ if $\alpha \in K^+$.
We say that an $\co$-lattice $L$ is \emph{positive definite} if $Q(v) \in K^+$ for all nonzero $v \in L$.
We will always assume that an $\co$-lattice is positive definite.
An $\co$-lattice $L$ is \emph{$n$-universal} if it represents every $\co$-lattice of rank $n$.

Given an $\co$-lattice $L$, the \emph{localisation of $L$ at $\fp$} is an $\co[\fp]$-lattice $L_\fp := L \otimes \co[\fp]$ which inherits a bilinear form and quadratic form extending those on $L$.
If an $\co$-lattice $L$ is positive definite (which we assume), then $L_\fp$ is nondegenerate.
Therefore, we will always assume that an $\co[\fp]$-lattice is nondegenerate.
Note that any $\co[\fp]$-lattice is free.

Let $L$ and $L'$ be $\co$-lattices.
We say that $L'$ is represented by $L$ over $\co[\fp]$ if $L'_\fp \to L_\fp$.
If this holds for all finite places $\fp$ of $\co$, we say that $L'$ is locally represented by $L$.
Moreover, we say that $L'$ is locally isometric to $L$ if $L'_\fp \cong L_\fp$ for all finite places $\fp$ of $\co$.
An $\co$-lattice $L$ is locally $n$-universal if it locally represents every $\co$-lattice of rank $n$.

The genus of an $\co$-lattice $L$ is defined as the set of all $\co$-lattices on $L \otimes K$ that are locally isometric to $L$, denoted by
\[
 \gen L := \{ L' \subseteq L \otimes K \mid L'_\fp \cong L_\fp \text{ for all finite places }\fp\}\text.
\]
If $L'$ is locally represented by $L$, then $L'$ is represented by some $\co$-lattice $\Lambda \in \gen L$.
It is well known that $\gen L$ is finite up to isometry.
The number of distinct isometry classes included in $\gen L$ is called the class number of $L$ and will be denoted by $\cl_K (L)$.
The Siegel mass constant $\mass{K}{L}$ is defined as \cite{SiegelI}:
\[
    \mass{K}{L} := \sum_{\Lambda \in \gen(L)} \frac{1}{ \card{ \Aut(\Lambda) } },
\]
where the sum runs over all inequivalent lattices $\Lambda$ in the genus of $L$.

An $\co$-lattice $L$ is called decomposable if there exist $\co$-lattices $L_1$, $L_2$ such that $L = L_1 \mathbin{\perp} L_2$ and $L_1 \ne 0 \ne L_2$. Otherwise, it is \emph{indecomposable}.
It is clear that any $\co$-lattice can be written as an orthogonal sum of indecomposable $\co$-lattices.
Since we assume that any $\co$-lattice is positive definite, the components of an indecomposable splitting are unique \cite[Theorem~105:1]{OMeara}.

An $\co$-lattice $L$ is called additively decomposable if there exist $\co$-lattices $L_1$, $L_2$ such that $L \subseteq L_1 \mathbin{\perp} L_2$ and $\Pr_1(L) \ne 0 \ne \Pr_2(L)$, where $\Pr_i$ denotes the canonical orthogonal projection from $L_1 \mathbin{\perp} L_2$ onto $L_i$ for $i = 1$, $2$. Otherwise, it is called \emph{additively indecomposable}. It is clear that a decomposable $\co$-lattice is additively decomposable. The converse also holds for unimodular lattices. Hence, a unimodular $\co$-lattice is decomposable if and only if it is additively decomposable (e.g. see \cite[p.~275]{Plesken}).

\section{Locally $n$-universal $\co$-lattices}  \label{sec:localstuff}

When $(K, \co) = (\Q, \Z)$, the lattice $I_{n+3}$ is locally $n$-universal, giving the upper bound $\MRU{K}{n} \le (n+3) \cl_K (I_{n+3})$. However, for totally real number fields $K$, the lattice $I_{n+3}$ is in general no longer locally $n$-universal, as illustrated in Example~\ref{not loc n univ}. Nevertheless, if there exists a finite list of unimodular $\co$-lattices of rank $m$ such that each $\co$-lattice of rank $n$ is represented by a lattice in the list, then $m \ge n+3$ (Lemma~\ref{min rank rep uni global}). Moreover, each $\co$-lattice of rank $n$ is represented by a unimodular $\co$-lattice on the quadratic $K$-space $I_{n+3}\otimes K$ (Proposition~\ref{dyadic In+3}). The number of unimodular genera on $I_{n+3}\otimes K$ is bounded above by a constant exponential in $d = [K:\Q]$ independently of the rank $n$ (Proposition~\ref{upper bound num genera}).

\begin{example}\label{not loc n univ}
 Suppose that $I_5$ is locally $2$-universal. Then $I_5$ is locally $1$-universal. Then by \cite[Theorem~3]{HsiaKitaokaKneser1978}, an integer in $\co$ of sufficiently large norm (or trace) must be represented by $I_5$ globally. However, if $K = \Q(\sqrt{m})$ for squarefree positive $m\equiv 2$, $3\pmod4$, then any integer in $\Z[\sqrt{m}]$ of the form $a + b\sqrt{m}$, where $a\in\Z$ and $b$ is odd, cannot be represented by any sum of squares, which gives rise to a contradiction.
\end{example}

Let $\fp$ be a finite place on $\co$.
We fix a prime $\pi_\fp \in \fp \setminus \fp^2$ and a nonsquare unit $\Delta_\fp$ in $\co$ that has the form $\Delta_\fp = 1 - 4\rho_\fp$ for some unit $\rho_\fp$ in $\co[\fp]$.
Recall that we denote by $K_\fp$ and $\co[\fp]$ the completion of $K$ and $\co$ at $\fp$, respectively.
Let $L$ be a (nondegenerate) $\co[\fp]$-lattice. We call $L$ \emph{classically maximal} if and only if, for any $\co[\fp]$-lattice $M$, $L \subseteq M \subset FL$ implies that $M = L$. Clearly, for any $\co[\fp]$-lattice $L$, there exists a classically maximal $\co[\fp]$-lattice $M$ such that $L \subseteq M \subset FL$.

\begin{lemma}\label{classify cla max}
Let $L$ be an $\co[\fp]$-lattice. $L$ is classically maximal if and only if one of the following conditions hold:

\begin{enumerate}[\textup{(\roman{enumi})}]
\item $\operatorname{ord}_\fp dL \le 1$;
\item when $\fp$ is nondyadic, $L \cong L_0 \mathbin{\perp} \left< \pi_\fp, -\pi_\fp \Delta_\fp \right>$, where $L_0$ is unimodular or zero.
\end{enumerate}
\end{lemma}

\begin{proof}
The proof is quite straightforward, so it is left to the reader.
\end{proof}

\begin{lemma}\label{min rank rep uni local}
Let $n$ be a fixed positive integer. Consider the following statement $P(m)$: any $\co[\fp]$-lattice of rank $n$ is represented by some unimodular $\co[\fp]$-lattice of rank $m$. If $n=1$ or $\fp$ is dyadic, then $P(m)$ holds if and only if $m \ge n+1$. Otherwise, $P(m)$ holds if and only if $m \ge n+2$.
\end{lemma}

\begin{proof}
Let $N$ be $n+1$ if $n=1$ or $\fp$ is dyadic, and $n+2$ otherwise.
Clearly, $P(m)$ implies $P(m+1)$.
Hence, it suffices to show that $\neg P(N-1)$ and $P(N)$.
First, we prove $P(N)$.
Clearly, $P(m)$ is equivalent to the statement $P'(m)$: any classically maximal $\co[\fp]$-lattice of rank $n$ is represented by some unimodular $\co[\fp]$-lattice of rank $m$. For any unimodular $\co[\fp]$-lattice $L_0$ and units $\epsilon_1$, \dots, $\epsilon_s$ in $\co[\fp]$, we have
\[
 L_0 \mathbin{\perp} \langle \pi\epsilon_1, \dots, \pi\epsilon_s \rangle \to L_0 \mathbin{\perp} \begin{pmatrix} \pi\epsilon_1 & 1 \\ 1 & 0 \end{pmatrix} \mathbin{\perp} \dotsb \mathbin{\perp} \begin{pmatrix} \pi\epsilon_s & 1 \\ 1 & 0 \end{pmatrix}\text.
\]
By Lemma~\ref{classify cla max} and the above representation, we conclude that $P'(N)$.

Now, we prove that $\neg P(N-1)$.
By discriminant arguments, $\langle 1, \dots, 1, \pi_\fp \rangle$ is not represented by any unimodular lattice of rank $n$.
Hence, we have $\neg P(n)$.
This completes the proof for the case where $n=1$ or $\fp$ is dyadic.
Now, assume that $n\ge 2$ and $\fp$ is nondyadic.
It suffices to prove that $\neg P(n+1)$.
Consider a lattice $L$ of the form $L \cong L_0 \mathbin{\perp} \langle \pi_\fp, -\pi_\fp \Delta_\fp \rangle$, where $L_0$ is a unimodular lattice or zero.
Suppose, by contradiction, that $M$ is a unimodular lattice of rank $n+1$ that represents $L$.
Since $M$ represents $L_0$, we may write $M \cong L_0 \mathbin{\perp} M_0$ for some unimodular sublattice $M_0$ of $M$.
Then $M_0$ must represent $\langle \pi_\fp, -\pi_\fp \Delta_\fp \rangle$, which implies $FM_0 \cong \langle \pi_\fp, -\pi_\fp \Delta_\fp, -\Delta_\fp dM_0 \rangle$, an anisotropic space.
On the other hand, since $M_0$ is unimodular of $\rank M_0 = 3$, $FM_0$ is isotropic.
This contradiction completes the proof.
\end{proof}

\begin{lemma}\label{min rank rep uni global}
Let $n$ be a fixed positive integer.
\begin{enumerate}[\textup{(\alph{enumi})}]
\item Let $m \le n+1$ be a positive integer. For any finite set $S$ of $\co$-lattices of rank $m$, there exists an $\co$-lattice of rank $n$ that is not represented by any lattice in $S$.
\item Let $n\ge 3$ and let $m \le n+2$ be a positive integer. For any finite set $S$ of unimodular $\co$-lattices of rank $m$, there exists an $\co$-lattice of rank $n$ that is not represented by any lattice in $S$.
\end{enumerate}
\end{lemma}

\begin{proof}
(a) We claim that the lemma (a) for $m$ implies the lemma (a) for $m-1$. Assume that there exists a finite set $S$ of $\co$-lattices $m-1$ such that any $\co$-lattice of rank $n$ is represented by some lattice in $S$. Then the set $\{ L \mathbin{\perp} \langle 1 \rangle \mid L \in S \}$ satisfies the analogous property for $m$, which proves the contrapositive of our claim. Hence, it suffices to prove the lemma for $m = n+1$. Moreover, we may assume that each lattice in $S$ is nondegenerate.

First, we assume $n=1$. Then $m=2$. Suppose that
\[
 S \otimes K = \{L \otimes K \mid L\in S\} = \{\langle a_1, a_1 d_1 \rangle, \dots, \langle a_s, a_s d_s \rangle\}\text,
\]
where $a_i$, $d_i \in K^+$ for $1\le i\le s$. Let $H$ be the Hilbert class field of $K$. Then $H$ is also totally real. Define $E := H(\sqrt{-d_1}, \dots, \sqrt{-d_s})$ and consider any embedding $\sigma: E \hookrightarrow \mathbb{C}$. Clearly, $\sigma \in \operatorname{Gal}(E/K)$, $\sigma |_H = \operatorname{id}_H$, and $\sigma(\sqrt{-d_i}) = -\sqrt{-d_i}$ for $1\le i\le s$. Hence, by Chebotarev density  (e.g. see \cite[p.~545]{Neukirch}), there are infinitely many nondyadic prime elements $\pi$ in $K$ such that for $1\le i\le s$, $-d_i$ is a nonsquare unit at $\pi$ and $a_i$ is a unit at $\pi$. Hence, by Pigeonhole Principle, there exist two such prime elements $\pi$ and $\pi'$ in $K$ of the same signature. Let $\Delta_\pi$ be any nonsquare unit at $\pi$. Then clearly, any quadratic $K$-space in $S \otimes K$ is isometric to $\langle 1, -\Delta_\pi \rangle$ as $K_\pi$-spaces. Therefore, $\langle \pi \pi' \rangle$ is a quadratic $K$-space that is not represented by any space in $S \otimes K$. Hence, any lattice on the space is not represented by any lattice in $S$, as required.

Next, we assume $n\ge 2$. There exists a nondyadic prime element $\pi$ in $K$ such that for all lattices $L\in S$, $L\otimes \co[\pi]$ is unimodular. Now, by Lemma~\ref{min rank rep uni local}, there exists an $\co[\pi]$-lattice of rank $n$ that is not represented by any unimodular $\co[\pi]$-lattice of rank $n+1$. Hence, there exists an $\co$-lattice of rank $n$ that is not represented by any lattice in $S$.

(b) Thanks to the Humbert reduction theory (e.g. see \cite{Humbert1940, ChanOh2023}), we know that there are only finitely many unimodular $\co$-lattices of rank $m$ up to isometry. Let $t$ be the number of such isometry classes $\{L_1, \ldots, L_t\}$. Fix $t$ distinct nondyadic places $\fp_1$, \dots, $\fp_t$ on $K$ such that $-1$ is nonsquare at $\fp_i$. 
For each $i$ with $1\le i\le t$, fix a prime element $\pi_i \in \fp_i \setminus \fp_i^2$.
By the Weak Approximation Theorem \cite[Theorem~11:8]{OMeara}, we can find $\alpha$, $\beta\in \co^+$ such that $\alpha \equiv -dL_i \pmod{\fp_i}$ and $\beta\equiv \pi_i \pmod{\fp_i^2}$ for $1\le i\le t$.
Then the $\co$-lattice $\left< 1, \ldots, 1, \alpha, \beta, \beta\right>$ of rank $n$ is not represented by any lattice in $S$.
\end{proof}

\begin{lemma}\label{dyadic uni}
Let $\fp$ be dyadic and let $V$ be a nondegenerate quadratic space over $K_\fp$. Then $\operatorname{ord}_\fp dV \subset 2\Z$ if and only if $V$ admits a unimodular $\co[\fp]$-lattice on it.
\end{lemma}
\begin{proof}
The if part is trivial. Now, suppose that $\operatorname{ord}_\fp dV \subset 2\Z$ and we aim to prove the only if part by induction on $\dim V$. If $\dim V = 1$, then $\epsilon \in dV$ for some $\epsilon \in \co[\fp]^\times$, and $V$ represents a unimodular lattice $\langle \epsilon \rangle$. Now, assume that $\dim V \ge 2$. If $V$ represents a unit in $\co[\fp]$, then we are done by inductive hypothesis. Otherwise, $V \cong \langle \pi_\fp, -\pi_\fp \Delta_\fp \rangle$ and it admits a unimodular lattice $\left(\begin{smallmatrix}\pi_\fp & 1\\1 & 4\pi_\fp^{-1}\rho_\fp\end{smallmatrix}\right)$.
\end{proof}

\begin{prop} \label{dyadic In+3}
Let $\fp$ be dyadic. Then any nondegenerate $\co[\fp]$-lattice of rank $n$ is represented by some unimodular $\co[\fp]$-lattice on $I_{n+3} \otimes K_\fp$.
\end{prop}
\begin{proof}
Let $L$ be any nondegenerate $\co[\fp]$-lattice of rank $n$. Clearly, we may assume that $L \subset I_{n+3} \otimes K_\fp$ and $L$ is classically maximal. First, we assume that $L$ is unimodular. Then by Lemma~\ref{dyadic uni}, $L^\perp$ admits a unimodular $\co[\fp]$-lattice, say $M$. Then $L \mathbin{\perp} M$ is a unimodular $\co[\fp]$-lattice on $I_{n+3} \otimes K_\fp$ that represents $L$, as required.

Now, we are left with the case where $L = L_0 \mathbin{\perp} \co[\fp]v$, where $L_0$ is unimodular or zero, and $Q(v) = \pi_\fp\epsilon$ with $\epsilon \in \co[\fp]^\times$. Let $U$ be the orthogonal complement of $L$ in $I_{n+3} \otimes K_\fp$. Since $U$ has dimension at least $3$, there exists $v'\in U$ such that $Q(v') = \pi_\fp \epsilon'$ for some $\epsilon' \in \co[\fp]^\times$. By Lemma~\ref{dyadic uni}, the orthogonal complement of $v'$ in $U$ admits a unimodular $\co[\fp]$-lattice $M$ on it. Since $\co[\fp]/\fp$ is perfect, there exists $\eta \in \co[\fp]^\times$ such that $\epsilon + \eta^2 \epsilon' \in \fp$. Now,
\[
 L_0 \mathbin{\perp} \left(\co[\fp] v + \co[\fp]\frac{v + \eta v'}{\pi_\fp}\right) \mathbin{\perp} M
\]
is a unimodular $\co[\fp]$-lattice on $I_{n+3} \otimes K_\fp$ that represents $L$, as required.
\end{proof}

\begin{lemma} \label{uni I n+3}
    Let $L$ be a $\co$-lattice of $\rank L=n$.
    Then $L$ is represented by a unimodular $\co$-lattice on $I_{n+3}\otimes K$.
\end{lemma}
\begin{proof}
    When $\fp$ is nondyadic, $L_\fp$ is represented by $I_{n+3} \otimes \co[\fp]$.
    When $\fp$ is dyadic, by Proposition \ref{dyadic In+3}, there is unimodular $\co[\fp]$-lattice $L'_\fp$ on $I_{n+3} \otimes K_\fp$ that represents $L_\fp$.
    By \cite[Proposition~81:14]{OMeara}, there is an $\co$-lattice $L'$ such that
    \[L' \otimes \co[\fp]= \begin{cases}
        I_{n+3} \otimes \co[\fp] & \text{when $\fp$ is nondyadic}\\
        L'_\fp  & \text{when $\fp$ is dyadic}.
    \end{cases}\]
    By construction, $L' \otimes \co[\fp]$ is unimodular for every finite place $\fp$ and 
    $L$ is locally represented by $L'$. Hence, there is a unimodular $\co$-lattice in $\gen(L')$ that represents $L$. 
\end{proof}

Let $L$ and $L'$ be unimodular $\co$-lattices on the quadratic space $I_n \otimes K$.
Recall that by \cite[Theorem~93:16]{OMeara}, for each dyadic prime $\fp$ on $K$, the completions $L_\fp$ and $L'_\fp$ are isometric if and only if the norm groups $\fg L_\fp$ and $\fg L'_\fp$ are the same (see below for the definition of norm groups).
Hence, $L$ and $L'$ are in the same genus if and only if the norm groups coincide at each dyadic prime.
Thus, the number of distinct genera of unimodular $\co$-lattices on the space $I_n \otimes K$ is bounded above by the product of the number of distinct norm groups at each dyadic prime, independently of the rank $n$.
We will estimate it in the following.

Let $\fp$ be dyadic.
Let $e_\fp = \ord_\fp 2$ and $f_\fp = \log_2 (\co[\fp]/\fp)$ denote the ramification index and inertial degree, respectively.
We have
\[
 \#(F^\times / (F^\times)^2) = 2^{e_\fp f_\fp + 2} \quad\text{and}\quad \#(\co[\fp]^\times / (\co[\fp]^\times)^2) = 2^{e_\fp f_\fp + 1}\text.
\]
For $\alpha$, $\beta\in F$, we write $\alpha \cong \beta$ if and only if $\alpha \in \beta (\co[\fp]^\times)^2$.
Let $\alpha\in F$.
The quadratic defect of $\alpha$ is defined by the following equation:
\[
 \fd(\alpha) := \bigcap\left\{\left(\alpha - \eta^2\right)\co[\fp] \mathrel{}\middle|\mathrel{} \eta\in F\right\}\text.
\]
For details on quadratic defects, we refer the reader to Section 63A of \cite{OMeara}.
By definition, for any $\eta\in F$, we have $\fd(\alpha + \eta^2) \subseteq \alpha \co[\fp]$.
In particular, $\fd(1 + \alpha) \subseteq \alpha \co[\fp]$ and $\fd(\alpha) \subseteq \alpha \co[\fp]$.
Also, we have $\fd(\alpha) = 0$ if and only if $\alpha \in F^2$, $\fd(\alpha) = \alpha \co[\fp]$ if and only if $\ord_\fp \alpha$ is odd, and $\fd(\alpha\beta^2) = \fd(\alpha)\beta^2$ for any $\beta\in F^\times$.
In particular, for any $\epsilon\in \co[\fp]^\times$, $\fd(\alpha\epsilon^2) = \fd(\alpha)$. Let $\epsilon\in \co[\fp]^\times$.
Then by \cite[Proposition~63:2]{OMeara}, $\fd(\epsilon)$ is among the following tower of ideals:
\[
 \fp \supset \fp^3 \supset \dotsb \supset 4\fp^{-3} \supset 4\fp^{-1} \supset 4\co[\fp] \supset 0\text.
\]
By \cite[Proposition~63:4]{OMeara}, $\fd(\epsilon) = 4\co[\fp]$ if and only if $\epsilon \cong \Delta$.
By \cite[Proposition~63:5]{OMeara}, if $\gamma\in \co[\fp]$ is such that $0 < \ord_\fp \gamma < 2e_\fp$ and that $\ord_\fp \gamma$ is odd, then $\fd(1+\gamma) = \gamma \co[\fp]$.
Hence, for each $\epsilon \in \co[\fp]^\times$, there exists $\gamma \in \co[\fp]$ such that $\epsilon \cong 1+\gamma$ and $\fd(\epsilon) = \gamma \co[\fp]$.
For $\alpha$, $\beta\in F^\times$ and for any fractional ideal $\mathfrak{a}$, we write $\alpha \cong \beta \mod{\mathfrak{a}}$ if and only if $\alpha/\beta \in \co[\fp]^\times$ and $\alpha - \beta\epsilon^2 \in \mathfrak{a}$ for some $\epsilon \in \co[\fp]^\times$, equivalently if and only if $\alpha/\beta \in \co[\fp]^\times$ and $\fd(\alpha/\beta) \subseteq \mathfrak{a}/\beta$.

Let $\fp$ be dyadic and let $L$ be an $\co[\fp]$-lattice.
The norm group of $L$ is the additive subgroup of $F$ defined by the equation $\fg L = Q(L) + 2\fs L$.
For details on norm groups and their properties, we refer the reader to Section 93A of \cite{OMeara}.
Any element of the norm group that has the least $\fp$-order is called a norm generator of $L$.
The weight ideal of $L$ is defined by the equation $\fw L = \fp \mathfrak{m} L + 2\fs L$, where $\mathfrak{m} L$ is the greatest ideal included in $\fg L$.
Then either $\ord_\fp \fn L + \ord_\fp \fw L$ is odd or $\fw L = 2\fs L$ (e.g. see~\cite[p.~254]{OMeara}).
Let $a$ be a norm generator of $L$ and $\fw L$ be the weight ideal of $L$.
Then by \cite[Proposition~93:4]{OMeara}, the norm group of $L$ is of the form
\[
 \fg L = a \co[\fp]^2 + \fw L\text.
\]
If $a'$ is another scalar such that $a \cong a'\pmod{\fw L}$, then $a'$ is also a norm generator of $L$, and vice versa (see \cite[Example~93:6]{OMeara}).
Hence, a norm group is determined by the pair of the weight ideal and the square class modulo weight ideal of the norm generators.

Let $\fp$ be dyadic and let $C = \{0, 1, \epsilon_1, \dots, \epsilon_{2^{f_\fp}-2}\}$ be a complete set of representatives of $\co[\fp]$ modulo $\fp$.
Let $A = \{\alpha(\alpha+1) \mid \alpha\in \co[\fp]\}$.
Then clearly, $1+4\alpha$ is a square if and only if $\alpha\in A$.
If $\alpha\in A$ and $\beta \equiv \alpha\pmod{\fp}$, then $\beta\in A$.
Hence, we have $\fp \subseteq A \subset \co[\fp]$.
Since $A/\fp A$ is an index $2$ subgroup of $\co[\fp]/\fp$, so is $A$ a subgroup of $\co[\fp]$.

\begin{lemma}\label{explicit usc}
Fix any $\rho' \in \co[\fp] \setminus A$. The following set
\[
 S := \{1 + c_1 \pi + c_2 \pi^3 + \dotsb + c_e \pi^{2e-1} + 4k\rho' \mid c_i \in C\text{, }k\in\{0, 1\}\}
\]
is a complete set of representatives of unit square classes.
\end{lemma}

\begin{proof}
Since $\# S = 2^{e_\fp f_\fp + 1}$, it suffices to show that no two elements in $S$ are in the same square class. Hence, we assume that $\alpha$ and $\beta$ are distinct elements of $S$ and prove that $\alpha\beta$ is nonsquare. We may assume that neither $\alpha$ nor $\beta$ is a square. Then, either $\alpha - 1$ or $\beta - 1$ has odd $\fp$-order. Hence, so is $\alpha\beta - 1$, proving that $\alpha\beta$ is nonsquare.
\end{proof}

\begin{lemma}\label{count norm group}
We have the following:
\begin{enumerate}[\textup{(\alph{enumi})}]
\item Let $v$ be a positive integer. The number of distinct unit square classes modulo $\fp^v$ is equal to
\[
 \frac{\co[\fp]^\times}{(\co[\fp]^\times)^2(1+\fp^v)} = \frac{1+\fp}{(1+\fp)^2(1+\fp^v)} = \begin{cases}
  2^{\lfloor \frac{v}2 \rfloor f_\fp} & \text{if $v \le 2e_\fp$,}\\
  2^{e_\fp f_\fp + 1} & \text{if $v > 2e_\fp$.}
 \end{cases}
\]
\item Let $0 \le u \le e_\fp$. The number of distinct subgroups of $\fp^u$ that can be realised as a norm group of a unimodular lattice $L$ with $\fn L = \fp^u$ is at most
\[
 \sum_{y=0}^{\left\lfloor \frac{e_\fp - u}2 \right\rfloor} 2^{y f_\fp} = \frac{2^{\left(\left\lfloor \frac{e_\fp - u}2 \right\rfloor + 1\right) f_\fp} - 1}{2^{f_\fp} - 1}\text.
\]
\item The number of distinct subgroups of $\co[\fp]$ that can be realized as a norm group of a unimodular lattice is at most
\[
 \sum_{x=0}^e \sum_{y=0}^{\left\lfloor \frac{x}2 \right\rfloor} 2^{y f_\fp} = \frac{2^{\left(\left\lfloor \frac{e_\fp}2 \right\rfloor + 2\right) f_\fp} - 2^{f_\fp}}{\left( 2^{f_\fp} - 1 \right)^2} + \frac{2^{\left(\left\lfloor \frac{e_\fp - 1}2 \right\rfloor + 2\right)f_\fp} - 2^{f_\fp}}{\left( 2^{f_\fp} - 1 \right)^2} - \frac{e_\fp + 1}{2^{f_\fp} - 1}\text.
\]
\end{enumerate}
\end{lemma}

\begin{proof}
(a) If $v > 2e_\fp$, then the lemma follows from Local Square Theorem \cite[Theorem~63:1]{OMeara}. Assume $v \le 2e_\fp$. Let $S$ be the complete set of representatives of unit square classes as in Lemma~\ref{explicit usc}. Let $\alpha$ and $\beta$ be distinct elements of $S$. If $\alpha - \beta \in \fp^v$, then $\alpha/\beta \in 1+\fp^v$, proving that $\fd(\alpha/\beta) \subset \fp^v$. If $\alpha - \beta \notin \fp^v$, then $\alpha/\beta - 1$ has an odd $\fp$-order less than $v$, which implies that $\fd(\alpha/\beta) \supsetneq \fp^v$. This proves the lemma.

(b) The weight ideal $\fw L$ has the form $\fp^{u+v}$, where
\[
 v \in \left\{v\text{ odd} \mathrel{}\middle|\mathrel{} 0\le v < e_\fp - u \right\} \cup \left\{ e_\fp - u \right\}\text.
\]
For each weight ideal above, by (a), among elements of $\fp$-order $u$, we have $2^{\lfloor \frac{v}2 \rfloor f_\fp}$ unit square classes modulo the weight ideal. This proves the lemma.

(c) This follows from (b).
\end{proof}

\begin{prop}\label{upper bound num genera}
For each dyadic prime $\fp$ on $K$, let $e_\fp = \ord_\fp 2$ and $f_\fp = \log_2 (\co[\fp]/\fp)$ denote the ramification index and inertial degree at $\fp$, respectively. The number of distinct unimodular genera on $I_n \otimes K$ is at most
\[
 \prod_{\fp \mid 2} \sum_{x=0}^{e_\fp} \sum_{y=0}^{\left\lfloor \frac{x}2 \right\rfloor} 2^{yf_{\fp}}\text,
\]
which is bounded above by
\[
 B = \begin{cases}
  5^{\frac{d}2} & \text{if $d$ is even,}\\
  2\cdot 5^{\frac{d-1}2} & \text{if $d$ is odd,}
 \end{cases}
\]
where $d = [K : \Q]$ is the degree of extension.
\end{prop}

\begin{proof}
For positive integers $a$ and $b$, define $G(a, b) := \sum_{x=0}^a \sum_{y=0}^{\lfloor \frac{x}2 \rfloor} 2^{yb}$. By Lemma~\ref{count norm group}, clearly the number of distinct unimodular genera is at most $\prod_{\fp \mid 2}G(e_\fp, f_\fp)$. Hence, it suffices to maximize $\prod_{i=1}^g G(a_i, b_i)$ where $a_1$, $b_1$, \dots, $a_g$, $b_g$ run over positive values under the constraint $\sum_{i=1}^g a_i b_i = d$.

Let $a$ and $b$ be any positive integers. First, we prove that $G(a, b) \le G(ab, 1)$. Let $T(ab)$ denote the region in $\mathbb{R}^2$ bounded by $y=0$, $y=\frac{x}2$, and $x=ab$. Then,
\[
 G(a, b) = \sum_{(x, y) \in T(ab) \cap (b\Z)^2} 2^y \le \sum_{(x, y) \in T(ab) \cap \Z^2} 2^y = G(ab, 1)\text.
\]
Hence, it suffices to maximize $\prod_{i=1}^g G(a_i, 1)$, where $(a_1, \dots, a_g)$ is a partition of $d$. Since $G(1, 1) = 2$ and $G(2, 1) = 5$, we have $(\ast)$ $G(1, 1)^2 \le G(2, 1)$. We claim that $(\ast\ast)$ $G(a+2, 1) \le G(2, 1)G(a, 1)$. Let $\mu$ and $\nu$ be integers such that $\mu + \nu = a+3$ and $\mu \le \nu \le \mu+1$. Then, $G(a, 1) = 2^\mu + 2^\nu - \mu - \nu - 2$. Hence, it suffices to show that
\begin{align*}
 0 & \le G(2, 1)G(a, 1) - G(a+2, 1)\\
 & = 5(2^\mu + 2^\nu - \mu - \nu - 2) - 2^{\mu+1} - 2^{\nu+1} + (\mu+1) + (\nu+1) + 2\\
 & = 3\cdot 2^\mu + 3\cdot 2^\nu - 4\mu - 4\nu - 6\text,
\end{align*}
which clearly holds from the fact that $\mu$, $\nu\ge 2$. This proves the claim $(\ast\ast)$. Now, let $\lambda = (a_1, \dots, a_g)$ be any partition of $d$. Let $(b_1, \dots, b_s)$ be a partition of $d$ that is obtained from $\lambda$ by exchanging each $a_i$ by $(2, \dots, 2)$ or $(2, \dots, 2, 1)$ according as each $a_i$ is even or odd. Let $(c_1, \dots, c_t)$ be a partition of $d$ that is either $(2, \dots, 2)$ or $(2, \dots, 2, 1)$ according as $d$ is even or odd. Then, by $(\ast)$ and $(\ast\ast)$,
\[
 \prod_{i=1}^g G(a_i, 1) \le \prod_{i=1}^s G(b_i, 1) \le \prod_{i=1}^{t} G(c_i, 1).
\]
The rightmost side equals $5^{\frac{d}2}$ or $2\cdot 5^{\frac{d-1}2}$ according as $d$ is even or odd, as required.
\end{proof}

\section{Indecomposable $\co$-lattices} \label{sec:indecomposables}

The goal of this section is to obtain lower and upper bounds of $\MRU Kn$ in terms of the number of indecomposable lattices of rank $r \le n$, which will be denoted $\IUL Kr$ in the sequel.

Let $n$ be a positive integer. Let $\mathcal{I}(n)$ be a complete set of representatives of indecomposable unimodular $\co$-lattices of rank $n$ modulo isometry. Clearly, $\mathcal{I}(n)$ is finite for any $n$. Let $\overline{\mathcal{I}}(n) = \mathcal{I}(1) \cup \dotsb \cup \mathcal{I}(n)$. For an $\co$-lattice $L$, we write $nL := L\mathbin{\perp}\dotsb\mathbin{\perp}L$ ($n$ times).

\begin{lemma} \label{bounds}
For a nondegenerate $\co$-lattice $L$ and a positive integer $m$, let $t_L(m)$ be the maximal nonnegative integer $t$ such that $t L \to I_m\otimes K$. Then, clearly $t_L(m) \le \lfloor \frac{m}{\rank L} \rfloor$. We have
\[
 \sum_{L\in \overline{\mathcal{I}}(n)}{\left\lfloor\frac{n}{\rank L}\right\rfloor} (\rank L) \le \MRU{K}{n} \le \sum_{L\in \overline{\mathcal{I}}(n+3)} t_L(n+3) \cdot (\rank L)\text.
\]
\end{lemma}

\begin{proof}
For the lower bound, let $M$ be any $n$-universal lattice. It suffices to show that $M$ represents
\[
 \bigperp_{L\in \overline{\mathcal{I}}(n)} \left\lfloor\frac{n}{\rank L}\right\rfloor L\text.
\]
Let $L\in \overline{\mathcal{I}}(n)$. Since $M$ is $n$-universal, $M$ represents ${\left\lfloor\frac{n}{\rank L}\right\rfloor} L$. Since $L$ is unimodular, a sublattice $M_L$ of $M$ isometric to ${\left\lfloor\frac{n}{\rank L}\right\rfloor} L$ splits $M$. Then, by the uniqueness of the orthogonal decomposition of $M$ into indecomposables \cite[Theorem~105:1]{OMeara}, $M$ is split by the sublattice $\bigperp_{L\in \overline{\mathcal{I}}(n)} M_L$.

For the upper bound, by Lemma~\ref{uni I n+3}, it suffices to show that
\[
 \bigperp_{L\in \overline{\mathcal{I}}(n+3)} \left\lfloor\frac{n+3}{\rank L}\right\rfloor L
\]
represents every unimodular lattice on $I_{n+3} \otimes K$. Let $N$ be any unimodular lattice on $I_{n+3} \otimes K$. Let $N = N_1 \mathbin{\perp} \dotsb \mathbin{\perp} N_r$ be the orthogonal decomposition of $N$ into indecomposables. Since $N$ is unimodular, so is $N_i$ for each $i$. Hence, each $N_i$ is isometric to some $L\in \overline{\mathcal{I}}(n+3)$. Thus, $N \cong \bigperp_{L\in \overline{\mathcal{I}}(n+3)}a_L L$ for some integers $a_L$ such that $0 \le a_L \le t_L(n+3)$. Hence, $N$ is represented by the given lattice.
\end{proof}

Let $n$ be a positive integer, and let $A(n)$ and $B(n)$ be the lower and upper bounds from the statement in the lemma above. Define $\IUL Kn$ as the number of indecomposable unimodular lattices of rank $n$ up to isometry. Note that $\card{ \mathcal{I}(n) } = \IUL Kn$. We first estimate the lower bound $A(n) = \sum_{r=1}^n {\left\lfloor\frac{n}{r}\right\rfloor} r \IUL Kr$. Clearly,
\[
 {\left\lfloor\frac{n}{r}\right\rfloor} r\begin{cases}
  = n & \text{if }r \mid n,\\
  \ge \max\{r, n-r+1\} & \text{if }r \nmid n\text,
 \end{cases}
\]
which implies that
\[
 A(n) \ge n\IUL K1 + \sum_{r=2}^{\lfloor\frac{n}{2}\rfloor} (n-r+1)\IUL Kr + \sum_{r = \lfloor\frac{n}{2}\rfloor + 1}^n r\IUL Kr\text.
\]
The upper bound clearly satisfies $B(n) \le \sum_{r=1}^{n+3} {\left\lfloor\frac{n+3}{r}\right\rfloor} r\IUL Kr \le (n+3)\sum_{r=1}^{n+3}\IUL Kr$. Finally, we estimate $B(n) - A(n)$. Note that
\[
 0 \le {\left\lfloor\frac{n+3}{r}\right\rfloor} r - {\left\lfloor\frac{n}{r}\right\rfloor} r\begin{cases}
  = 3 & \text{if }r = 1\text,\\
  \le 4 & \text{if }r = 2\text,\\
  \le r & \text{if }3 \le r \le {\left\lfloor\frac{n+3}{2}\right\rfloor}\text,\\
  = 0 & \text{if }{\left\lfloor\frac{n+3}{2}\right\rfloor} < r \le n\text.
 \end{cases}
\]
Hence,
\[
 B(n) - A(n) \le 3\IUL K1 + 4\IUL K2 + \sum_{r=3}^{\left\lfloor\frac{n+3}{2}\right\rfloor} r\IUL Kr + \sum_{r=n+1}^{n+3} r\IUL Kr\text.
\]
Therefore, so far, we estimated the lower bound, the upper bound, and the difference in terms of $\IUL Kr$. Define $\UL Kn$ as the number of unimodular lattices of rank $n$ up to isometry. Note that
\[
 \card{ \overline{\mathcal{I}}(n) } = \sum_{r=1}^n \IUL Kr \le \UL Kn\text.
\]
From the definitions, there is an apparent relationship between $\IUL Kn$ and the generating function of $\UL Kn$ as follows:
\[
 \prod_{n=1}^\infty (1 - X^n)^{-\IUL Kn} = 1 + \sum_{n=1}^\infty \UL Kn X^n\text.
\]
In other words, the sequence $\{\IUL Kn\}$ is the Inverse Euler transform \cite[p.~20]{SloanePlouffe} of the sequence $\{\UL Kn\}$.

\section{Bounds for the Siegel mass constant}  \label{sec:mass}

Let $L$ be a positive definite $\co$-lattice.  Recall the Siegel mass constant $\mass{K}{L}$ defined as \cite{SiegelI}:
\begin{equation} \label{eq:mass}
    \mass{K}{L} := \sum_{\Lambda \in \gen(L)} \frac{1}{ \card{ \Aut(\Lambda) } },
\end{equation}
where the sum runs over all inequivalent lattices $\Lambda$ in the genus of $L$.

Our main goal in this section is to use Siegel's mass constant to obtain bounds for the number of rank $n$ $\co$-lattices $\UL{K}{n}$.  In particular, given the trivial lower bound $ \card{ \Aut(\Lambda) } \geq 2$ for all $\co$-lattices $\Lambda$, we note from (\ref{eq:mass}) that we have the following lower and upper bounds,
\begin{equation*}
    2 \mass{K}{L} \leq \cl_K(L) \leq \mass{K}{L} \cdot \max_{\Lambda \in \text{gen}(L)} \card{ \Aut(\Lambda) },
\end{equation*}
where the class number $\cl_K(L)$ denotes the number of inequivalent $\co$-lattices lying in the genus of $L$.

\subsection{The standard lattice $I_n$} 
We first explicitly evaluate $\mass{K}{L}$ in the case where $L = I_n$. Here, we shall make use of the following theorem of K\"{o}rner \cite{Korner}, which gives a completely explicit computation of the mass constant $\massI{K}{n}$:

Let $\psi$ be the generating character of $K(\sqrt{-1})/K$, 
 that is, we have $\zeta_{K(\sqrt{-1})}(s) = L(s, \psi, K) \zeta_K(s)$, where $L(s, \psi, K)$ is the $L$-function
 \begin{equation*}
     L(s, \psi, K) = \prod_{\fp } \Big( 1 - \frac{\psi(\fp)}{\norm{\fp}^s} \Big)^{-1} \quad \text{for } \text{Re}(s) > 1.
 \end{equation*}
 We note that $\psi$ can be explicitly defined as follows \cite[p.~277]{Korner}:

\begin{equation*}
    \psi(\fp) = \begin{cases}
        \left(  \frac{-1}{\fp} \right)   & \text{if } \fp \nmid 2, \\
        1  &  \text{if } \fp \mid 2 \text{ and } x^2 \equiv -1 \pmod{\fp^{2e_{\fp} + 1}} \text{ for some } x \in \co, \\ %
        -1  &  \text{if } \fp \mid 2 \text{ and } x^2 \not \equiv -1 \pmod{\fp^{2e_{\fp} + 1}} \text{ for all } x \in \co, \\ %
        & \qquad \text{and } x^2 \equiv -1 \pmod{\fp^{2e_{\fp}}} \text{ for some } x \in \co  \\ %
        0  &  \text{if } \fp \mid 2 \text{ and } x^2 \not \equiv -1 \pmod{\fp^{2e_{\fp}}} \text{ for all } x \in \co.
    \end{cases}
\end{equation*}

\begin{theorem}[K\"{o}rner {\cite[p.~277]{Korner}}]  Let $K$ be a totally real field and let $n > 1$.  For each dyadic prime $\fp$ in $K$, let $e_\fp$ and $f_\fp$ be the ramification index and residue degree of $\fp$ respectively. Let $\xi(n, \fp)$ be defined as
\begin{equation*}
    \frac{\xi(n, \fp)}{\norm{\fp}^{(1-n) \lfloor e_{\fp}/2 \rfloor - e_{\fp}}} = \begin{cases}
     1 & \text{if } 2 \mid e_{\fp}, \\
    1/2 & \text{if } 2 \nmid e_{\fp}, n \equiv 2\pmod{4}, \\
    \dfrac{1}{2} \Big(  1 + \dfrac{(-1)^{f_{\fp} \floor{(n+1)/4} }}{\norm{\fp}^{ \floor{(n-1)/2 }} } \Big)  & \text{if } 2 \nmid e_{\fp}, n \not \equiv 2\pmod{4}.
    \end{cases}
\end{equation*}
Then
\begin{equation} \label{eq:kornermass}
    \massI{K}{n} = \sigma_n \left( \frac{2 \Gamma(\frac{1}{2}) \Gamma(\frac{2}{2}) \cdots \Gamma(\frac{n}{2})}{\pi^{n(n+1)/4}}   \right)^{[K :\Q]}  \Delta_K^{n(n-1)/4}  
    \prod_{\fp | 2} \xi(n, \fp) 
    \prod_{i=1}^{\lfloor (n-1)/2 \rfloor} \zeta_K(2i) ,
\end{equation}
where $\Gamma(s)$ is the usual Gamma function, $\zeta_K(s)$ is the Dedekind zeta function of $K$, and where
    \begin{equation} \label{eq:sigma}
        \displaystyle{
        \sigma_n  = \begin{cases}
            1 & \text{if } n \text{ odd}, \\
            \displaystyle{ \zeta_K(n/2) \prod_{\fp | 2} \Big(  1 - \frac{1}{\norm{\fp}^{n/2}} \Big) } & \text{if } n \equiv 0 \pmod{4}, \\
            \displaystyle{ \frac{\Res_{s=1} \zeta_{K(\sqrt{-1})}(s) }{\Res_{s=1} \zeta_K(s)} \prod_{\fp | 2}  \Big  (1 - \frac{\psi(\fp)}{\norm{\fp}} \Big) } & \text{if } n = 2, \\
            \displaystyle{ L(n/2, \psi, K)  \prod_{\fp | 2}  \Big(  1 - \frac{\psi(\fp)}{\norm{\fp}^{n/2}} \Big) } & \text{if } n > 2, n \equiv 2 \pmod{4}.
        \end{cases} }
    \end{equation}

\end{theorem}

\subsection{Asymptotics for large $n$}

First, we will prove an explicit asymptotic for the mass constant as $n \to \infty$:

\begin{theorem} \label{thm:massasymp}
    Let $K$ be a totally real degree $d$ field, and for each dyadic prime $\fp$ in $K$, let $e_\fp$ denote the ramification degree of $\fp$.  Then
    \begin{equation*}
    \massI{K}{n} \sim C_K \left( \frac{n}{2 \pi e \sqrt{e}} \right)^{d n^2/4} \left( \frac{8 \pi e }{ n}  \right)^{dn/4} \left( \frac{1}{n} \right)^{d/24}  \Delta_K^{n(n-1)/4} \left( \prod_{\fp | 2} \frac{1}{\norm{\fp}^{\floor{ e_{\fp}/2 }}}  \right)^n
    \end{equation*}
    as $n \to \infty$.  Here $C_K$ is the explicit constant
    \begin{equation*}
        C_K = 2^{-5d/4} e^{d/24} \exp{\Big(\frac{1}{12} - \zeta'(-1) \Big)}^{-d/2} \cdot  \left( \prod_{\fp | 2}   \frac{\norm{\fp}^{\floor{ e_{\fp}/2  } - e_{\fp}}}{2^{\delta_{2 \nmid e_{\fp}}}} \right) \cdot \prod_{i=1}^\infty \zeta_K{(2i)},
    \end{equation*}
    where $\delta_{2 \nmid e_{\fp}}$ is 1 if $e_{\fp}$ is odd, and 0 otherwise.
    
\end{theorem}

\begin{proof}
    We first observe that $\sigma_n \to 1$ as $n \to \infty$.  Indeed, this easily follows by definition, noting that there are only finitely many dyadic primes $\fp$ and the observations that $\zeta_K(n/2) \to 1$, $L(n/2, \psi,K) \to 1$, $(1 - 1/\norm{\fp}^{n/2}) \to 1$, and $(1 - \psi(\fp)/\norm{\fp}^{n/2}) \to 1$ as $n \to \infty$.

    Let's now consider the product of Gamma functions $\prod_{i=1}^n \Gamma(i/2)$. This gives us the biggest asymptotic contribution towards $\massI{K}{n}$.  By applying a theorem of Kellner \cite[Theorem~3.1]{Kellner}, we have the asymptotic
    \begin{equation*}
        \prod_{i=1}^n \Gamma(i/2) \sim \mathcal{C} \cdot \Big( \frac{n}{2 e^{3/2}}  \Big)^{n^2/4} \Big( \frac{8 \pi^2 e}{n} \Big)^{n/4}  \Big(  \frac{1}{n}  \Big)^{1/24}
    \end{equation*}
    as $n \to \infty$, where $\mathcal{C} \approx 0.774144 \dots$ is a constant given by $\mathcal{C} = 2^{-1/4} e^{1/24} \exp(1/12 - \zeta'(-1))^{-1/2}$.  

    For each dyadic prime $\fp$, we now consider the asymptotic behaviour of $\xi(n, \fp)$.  
    If $2 \mid e_\fp$, then $\xi(n, \fp) = \norm{\fp}^{(1-n) (e_{\fp}/2) - e_{\fp}}$ for all $n$.
    If $2 \nmid e_\fp$, then since $(-1)^{f_{\fp}} / \norm{\fp}^{\floor{(n-1)/2}} \to 0$ as $n \to \infty$ we thus have $\xi(n, \fp)/ \norm{\fp}^{(1-n) (e_{\fp}/2) - e_{\fp}} \to 1/2$ as $n \to \infty$.

    Thus, we have
    \begin{align*}
        \prod_{2 | \fp} \xi(n, \fp) &\to \bigg( \prod_{\substack{2 | \fp \\ 2 \mid e_{\fp}}} \norm{\fp}^{(1-n) (e_{\fp}/2) - e_{\fp}} \bigg) \cdot \bigg( \prod_{\substack{2 | \fp \\ 2 \nmid e_{\fp}}} \frac{\norm{\fp}^{(1-n) (e_{\fp}/2) - e_{\fp}}}{2} \bigg) \\
        &= \bigg( \prod_{\fp | 2}   \frac{\norm{\fp}^{\lfloor e_{\fp}/2 \rfloor - e_{\fp}}}{2^{\delta_{2 \nmid e_{\fp}}}} \bigg) \cdot \Big( \prod_{\fp | 2}  \frac{1}{\norm{\fp}^{\floor{e_\fp / 2}}}  \Big)^n
    \end{align*}
    as $n \to \infty$, where $\delta_{2 \nmid e_{\fp}}$ is 1 if $2 \nmid e_{\fp}$ and 0 otherwise.  
    Finally, it is clear that $\prod_{i=1}^{\lfloor (n-1)/2 \rfloor} \zeta_K(2i) \to \prod_{i=1}^{\infty} \zeta_K(2i)$ as $n \to \infty$, given that $\zeta_K(2i) \leq \zeta(2i)^{[K :\Q]}$ and the convergence of the infinite product $\prod_{i=1}^\infty \zeta(2i)$ (e.g. see \cite[Lemma~1.1]{Kellner}).  Putting all the terms together yields the desired asymptotic formula for $\massI{K}{n}$. 
\end{proof}

By thus taking the logarithm of $\massI{K}{n}$ in the above theorem, we can explicitly compute the first four main terms of the asymptotic expansion of $\log \massI{K}{n}$ as
\begin{align} \label{eq:logmassasymp}
    \log \massI{K}{n} &= \frac{d}{4} n^2 \log{n} + n^2 \bigg( \nquadcoeff \bigg) - \frac{d}{4} n \log{n} \\ %
    &\quad + n \bigg(  \nlinearcoeff \bigg) + \bigO{ \log{n} }. \nonumber
\end{align}

\textbf{Remark.}  We note that Theorem~\ref{thm:massasymp} is well-known in the case $K = \Q$, e.g. see Milnor--Husemoller \cite[p.~50]{MilnorHusemoller} or Finch \cite[p.~656]{FinchII}.

\subsection{Bounds for small $n$}

In this section, we will prove explicit bounds on $\massI{K}{n}$ for any $n \geq 3$.  We note that similar methods were also used by Pfeuffer \cite{Pfeuffer71, Pfeuffer} to obtain lower bounds for $\massI{K}{n}$.

\begin{theorem} \label{thm:massInbounds}
    Let $K$ be a totally real degree $d$ field, and let $n \geq 3$. Define $F_n$ as
    \begin{equation*}
        F_n := \frac{2 \Gamma(\frac{1}{2}) \Gamma(\frac{2}{2}) \cdots \Gamma(\frac{n}{2})}{\pi^{n(n+1)/4}}.
    \end{equation*}
    Then
    \begin{equation} \label{eq:massbound}
        \Big( \frac{F_n }{2^{(n+5)/2} \cdot (2 \zeta(n/2) )^{\delta_{2 \mid n}}} \Big)^d \leq  \frac{\massI{K}{n}}{\Delta_K^{n(n-1)/4}} \leq \Big( \frac{3F_n \zeta(n/2)^2}{2} \cdot \prod_{i=1}^\infty \zeta(2i) \Big)^d,
    \end{equation}
    where $\delta_{2 \mid n} = 1$ if $n$ is even, and 0 otherwise.
\end{theorem}

\begin{proof}
    First, we prove lower and upper bounds for $\sigma_n$ for all $n \geq 3$. By definition, $\sigma_n = 1$ if $n$ odd. Otherwise, if $n$ even, then note that we have the bounds $1 \leq \zeta_K(n/2) \leq \zeta(n/2)^d$ (e.g. see \cite{Conrad}) and 
    \begin{equation*}
        \frac{1}{\zeta(n/2)^d} \leq \frac{1}{\zeta_K(n/2)} \leq L(n/2, \psi, K) \leq \zeta_{K(\sqrt{-1})}(n/2)  \leq \zeta(n/2)^{2d}.
    \end{equation*}
    As $\norm{\fp}^{n/2} \geq 2$, and at most $d$ dyadic primes $\fp$ divide 2, this also gives the bounds
    \begin{equation*}
        \frac{1}{2^d} \leq \prod_{\fp | 2} \Big( 1 - \frac{1}{2} \Big) \leq \prod_{\fp | 2} \Big( 1 - \frac{\psi(\fp)}{\norm{\fp}^{n/2}} \Big) \leq \prod_{\fp | 2} \Big( 1 + \frac{1}{2} \Big)  \leq \Big( \frac{3}{2} \Big)^d.
    \end{equation*}
    Thus, for all even $n \geq 4$, we have the bounds
    \begin{equation*}
        \Big(  \frac{1}{2 \zeta(n/2)} \Big)^d  \leq \sigma_n \leq \Big( \frac{3 \zeta(n/2)^2}{2} \Big) ^d.
    \end{equation*}
    Now we estimate bounds for $\xi(n, \fp)$ for all dyadic primes $\fp$.  By definition, we clearly have
    \begin{equation*}
        \frac{1}{4} \leq \frac{\xi(n, \fp)}{\norm{\fp}^{(1-n) \floor{e_\fp/2} - e_\fp}} \leq 1 
    \end{equation*}
    for all $n \geq 3$.
    Now note that, as $n \geq 3$, we have
    \begin{equation*}
        \norm{\fp}^{(1-n) (e_\fp/2) - e_\fp} \leq \norm{\fp}^{(1-n) \floor{e_\fp/2} - e_\fp} \leq \norm{\fp}^{-e_{\fp}}.
    \end{equation*}
    Noting that $\norm{\fp} = 2^{f_{\fp}}$ and $\sum_{\fp \mid 2} e_\fp f_\fp = d$, then taking the product over all dyadic primes $\fp$, we obtain the following upper and lower  bounds for $\prod_{\fp | 2} \xi(n, \fp)$:
    \begin{equation} \label{eq:xibounds}
        \Big( \frac{1}{2^{(n+5)/2}} \Big)^d \leq \frac{1}{4^d} \prod_{\fp | 2} 2^{f_\fp ((1-n) ( e_{\fp}/2) - e_{\fp} )} \leq \prod_{\fp | 2} \xi(n, \fp) \leq  \prod_{\fp | 2} 2^{-f_{\fp} e_{\fp}} \leq \frac{1}{2^d}.
    \end{equation}

    Finally, we have the bound $ 1 \leq \prod_{i=1}^{\floor{(n-1)/2}} \zeta_K(2i) \leq \Big(\prod_{i=1}^{\floor{(n-1)/2}} \zeta(2i) \Big)^d$.  Putting everything together gives the bound in (\ref{eq:massbound}).
\end{proof}

\textbf{Remark}.  We note that Pfeuffer \cite[p.~373]{Pfeuffer71} similarly obtained the following lower bound for the mass of an arbitrary $\co$-lattice $L$ of rank $n \geq 3$:
\begin{equation*}
    \mass{K}{L} \geq \left( \frac{ \Gamma(\frac{1}{2}) \Gamma(\frac{2}{2}) \cdots \Gamma(\frac{n}{2})}{\pi^{n(n+1)/4}}   \right)^{[K :\Q]}  \Delta_K^{n(n-1)/4}  \Big( \frac{\norm{\fd L} }{\norm{\fs L}^n } \Big)^{1/26} \cdot \frac{1}{f(n)^{[K : \Q]}},
\end{equation*}
where $f(3) = 4 \cdot \sqrt[26]{8}$, $f(4) = 9 \cdot \sqrt[13]{2}$ and $f(n) = 2^{n-1}$ for $n \geq 5$.   Applying Pfeuffer's bound to the case of the unimodular lattice $L = I_n$, one obtains a similar lower bound to the one given above in Theorem~\ref{thm:massInbounds}, albeit slightly weaker, approximately by a factor of $2^{n [K:\Q]/2}$ for large $n$.
\bigskip

An immediate corollary of the above theorem, is that, for any fixed degree $d$ and $n \geq 3$, we have that $\massI{K}{n} \asymp_{n,d} \Delta_K^{n(n-1)/4}$ as $\Delta_K \to \infty$ over any sequence of totally real fields $K$ of fixed degree $d$. Thus, for any fixed constant $C$, there are only finitely many totally real degree $d$ fields $K$ such that $\massI{K}{n} < C$.  However, we can furthermore remove the condition on the degree $d$ using some discriminant bounds of Odlyzko.  

\begin{theorem}[Odlyzko {\cite{Odlyzko1977}}] Let $K$ be a totally real field of $d$ and discriminant $\Delta_K$.  Then $\Delta_K \geq (60.1)^d \cdot \exp(-254)$.
\end{theorem}

By thus substituting the above Odlyzko bound in (\ref{eq:massbound}), we obtain the following corollary, essentially reproving a finiteness result for the mass constant $\massI{K}{n}$ originally due to Pfeuffer \cite{Pfeuffer}:

\begin{cor}
     Let $K$ be a totally real degree $d$ field, and let $n \geq 3$. Then 
     \begin{equation*}
         \massI{K}{n} 
         \geq \Big(  \frac{F_n \cdot 60.1^{n(n-1)/4}}{2^{(n+5)/2} \cdot (2 \zeta(n/2) )^{\delta_{2 \mid n}} } \Big)^d \cdot \exp(-254 n(n-1)/4).
    \end{equation*}
    In particular, given any fixed $C > 0$, there are only finitely many totally real fields $K$ such that $\massI{K}{n} < C$, with effectively bounds on the degree $d$ and discriminant $\Delta_K$ for such fields $K$.
\end{cor}

\begin{proof}
    Given a totally real field $K$ of degree $d$, we have the Odlyzko bound of $\Delta_K \geq (60.1)^d \cdot \exp(-254)$ \cite{Odlyzko1977}.  This therefore gives the lower bound for $\massI{K}{n}$ as
    \begin{align*}
        \massI{K}{n} &\geq \Delta_K^{n(n-1)/4} \Big( \frac{F_n}{2^{(n+5)/2} \cdot (2 \zeta(n/2) )^{\delta_{2 \mid n}}} \Big)^d \\
        &\geq \Big(  \frac{F_n \cdot 60.1^{n(n-1)/4}}{2^{(n+5)/2} \cdot (2 \zeta(n/2) )^{\delta_{2 \mid n}}} \Big)^d \cdot \exp \big( \frac{-127 n(n-1)}{2} \big).
    \end{align*}
    For brevity, we define $\mathcal{M}_n := (F_n \cdot 60.1^{n(n-1)/4}) / ( 2^{(n+5)/2} \cdot (2 \zeta(n/2) )^{\delta_{2 \mid n}} )$.  It now suffices to check that the $\mathcal{M}_n$ is greater than 1 for all $n \geq 3$, to ensure that an upper bound on $\massI{K}{n}$ yields an upper bound on the degree $d$. Computing the first few terms of $\mathcal{M}_n$, we get

    \begin{table}[h]
    \centering
    \caption{Values of $\mathcal{M}_n$ for the first few integers $n \geq 3$ (given to two significant places)}
    \label{tab:Mn}
    \begin{tabular}{@{}cccccccc@{}}
    \toprule
    $n$ & \textbf{3} & \textbf{4} & \textbf{5} & \textbf{6} & \textbf{7} & \textbf{8} & \textbf{9}  \\ \midrule
    $\mathcal{M}_n$ &   $2.95$  & $29.94$  &  $1.91 \cdot 10^4$ &  $1.02 \cdot 10^7$  & $2.27 \cdot 10^{11}$  & $7.67 \cdot 10^{15}$ & $1.03 \cdot 10^{22}$  \\ \bottomrule
    \end{tabular}
    \end{table}

    As $\log \mathcal{M}_n \sim \frac{1}{4} n^2 \log n$ as $n \to \infty$, and with a computation of the first few terms, we can easily observe that $\mathcal{M}_n > 1$ for all $n \geq 3$.

    Thus, if $\massI{K}{n} < C$, then this gives the following upper bounds for the degree $d$ and the discriminant $\Delta_K$ of $K$, in terms of the constant $C$:
    \begin{equation}  \label{eq:dboundmass}
        d \leq \frac{\log (C \exp ( 127 n(n-1)/2 ) ) }{  \log ( (F_n \cdot 60.1^{n(n-1)/4} ) / ( 2^{(n+5)/2} \cdot (2 \zeta(n/2) )^{\delta_{2 \mid n}} ) ) } 
    \end{equation}
    and
    \begin{equation}  \label{eq:discboundmass}
        \Delta_K \leq \Big( \frac{C}{ (F_n / ( 2^{(n+5)/2} \cdot (2 \zeta(n/2) )^{\delta_{2 \mid n}} ) )^d } \Big)^{4 / (n(n-1))}.\qedhere
    \end{equation}
\end{proof}

\subsection{Bounds for $n = 2$}

In the case where $n = 2$, we can no longer show that $\massI{K}{n} \asymp_{n,d} \Delta_K^{n(n-1)/4}$ as $\Delta_K \to \infty$ over any sequence of totally real fields of degree $d$.  In particular, we need to take more care with the bounds of $\sigma_2$ given in (\ref{eq:sigma}).

In this section, we summarise the known bounds for $\massI{K}{n}$ in the case $n = 2$, both unconditional and conditional on GRH.  We note that we have the following result of Pfeuffer:

\begin{theorem}[Pfeuffer {\cite[Theorem~2]{Pfeuffer}}] \label{thm:pfeuffer}
    Let $C > 0$ and fix a positive integer $d$.  Then there are only finitely many totally real fields $K$ of degree $d$ such that $\cl_K(I_2) < C$.  Only finitely many of all these fields (of arbitrary degree) can be reached by a tower of relatively normal extensions from $\Q$.  There are finitely many such fields $K$ in all, provided either the generalized Riemann hypothesis or the Artin conjecture is true.
\end{theorem}

To provide some effective bounds in the case $n = 2$, we shall make use of known upper and lower bounds on $\Res_{s=1} \zeta_K(s)$ due to Stark and Louboutin.  First, we summarise the known bounds we know so far:

\begin{theorem}[Stark {\cite{Stark}}, Louboutin {\cite{Louboutin}}] \label{thm:dedekindbounds}
    Let $K$ be a degree $d$ number field with $d \geq 2$.  Then
    \begin{equation*}
        \frac{0.0014480}{ d \cdot d! \cdot |\Delta_K|^{1/d}} < \Res_{s=1} \zeta_K(s) \leq  \Big( \frac{e \log | \Delta_K |}{2 (d-1)} \Big)^{d-1}.
    \end{equation*}
\end{theorem}

\begin{proof}
    The lower bound is due to Stark \cite{Stark} (see also Garcia--Lee \cite{GarciaLee}) and the upper bound is due to Louboutin \cite{Louboutin}.
\end{proof}

We can now give a quick proof of Pfeuffer's results for the rank $n = 2$ case, and can further give effective discriminant bounds in the degree $d \geq  3$ case.

\begin{theorem}
    Let $C > 0$ and fix a positive integer $d \geq 1$.  Then there are only finitely many degree $d$ totally real fields $K$ such that $\massI{K}{2} < C$.  In particular, if $d \geq 3$ then we have the explicit upper bound
    \begin{equation*}
        \Delta_K \leq \max \left( \exp \big( (12(d-1))^2 \big), \; \Big (\frac{\massI{K}{2} \cdot (\pi \cdot 2^{5/2})^d \cdot d \cdot (2d)!   }{ 0.000181 \cdot ((d-1)/e)^{d-1} } \Big)^{\frac{12d}{5d-12}} \right).
    \end{equation*}
\end{theorem}

\begin{proof}
    We first compute a lower bound for $\sigma_2$.  By Theorem~\ref{thm:dedekindbounds}, we have 
    \begin{align*}
        \sigma_2 &= \frac{\Res_{s=1} \zeta_{K(\sqrt{-1})}(s) }{\Res_{s=1} \zeta_K(s)} \prod_{\fp | 2}  \Big  (1 - \frac{\psi(\fp)}{\norm{\fp}} \Big) \\
        &\geq \frac{1}{2^d} \cdot \Big( \frac{2(d-1)}{e \log  \Delta_K} \Big)^{d-1} \cdot \frac{0.0014480}{2d \cdot (2d)! \cdot | \Delta_{K(\sqrt{-1})} |^{1/2d}}.
    \end{align*}

Noting that the only primes ramifying in the relative quadratic extension $K(\sqrt{-1})/K$ are possibly the dyadic primes $\fp$ lying above 2, we have the bound
\begin{equation*}
    |\Delta_{K(\sqrt{-1})}| = |\Nm_{K/\Q} ( \Delta_{K(\sqrt{-1})/K} )| \cdot \Delta_K^2 \leq 2^{2d} \Delta_K^2.
\end{equation*}
This therefore gives the following lower bound for $\sigma_2$:
\begin{align*}
    \sigma_2 &\geq \frac{1}{2^d} \cdot \Big( \frac{2(d-1)}{e \log  \Delta_K} \Big)^{d-1} \cdot \frac{0.0014480}{2d \cdot (2d)! \cdot 2 \cdot \Delta_K^{1/d}} \\
    &= \Big( \frac{(d-1)}{e \log  \Delta_K} \Big)^{d-1} \cdot \frac{0.000181}{d \cdot (2d)! \cdot \Delta_K^{1/d}}.
\end{align*}
Using the bounds for $\xi(n, \fp)$ in (\ref{eq:xibounds}), we have
\begin{equation*}
    \prod_{\fp | 2} \xi(2, \fp) \geq \frac{1}{2^{7d/2}}.
\end{equation*}
Thus, by using K\"{o}rner's computation of $\massI{K}{2}$ given in (\ref{eq:kornermass}),we have
\begin{equation} \label{eq:masslowerI2}
    \massI{K}{2} \geq \Delta_K^{1/2} \Big (\frac{2 \Gamma( \frac{1}{2} ) \Gamma( \frac{2}{2} )}{\pi^{3/2} \cdot 2^{7/2}} \Big)^d \cdot \Big( \frac{(d-1)}{e \log  \Delta_K} \Big)^{d-1} \cdot \frac{0.000181}{d \cdot (2d)! \cdot \Delta_K^{1/d}}.
\end{equation}
By simplifying and isolating the discriminant terms, we obtain the upper bound
\begin{equation*}
    \frac{\Delta_K^{(d-2)/(2d)}}{(\log \Delta_K )^{d-1}} 
    \leq \frac{\massI{K}{2} \cdot (\pi \cdot 2^{5/2})^d \cdot d \cdot (2d)!   }{ 0.000181 \cdot ((d-1)/e)^{d-1} }.
\end{equation*}
To obtain an explicit upper bound on $\Delta_K$, we can note that, for any fixed $k \geq 2$, we have that $ x > (\log x)^k$ for all $x > \exp(k^2)$ (as $k > 2 \log k$).  Thus, we have $(\log \Delta_K)^{d-1} < \Delta_K^{1/12}$ if $\Delta_K > \exp((12(d-1))^2)$.  Thus, for each $d \geq 3$, this gives the upper bound on $\Delta_K$:
\begin{equation} \label{eq:discbound2}
    \Delta_K \leq \max \left( \exp \big( (12(d-1))^2 \big), \; \Big (\frac{\massI{K}{2} \cdot (\pi \cdot 2^{5/2})^d \cdot d \cdot (2d)!   }{ 0.000181 \cdot ((d-1)/e)^{d-1} } \Big)^{\frac{12d}{5d-12}} \right).\qedhere
\end{equation}

\end{proof}

\subsection{Arbitrary unimodular lattices $L$}
To effectively obtain an upper bound for $\UL{K}{n}$, we must obtain upper bounds on $\cl_K(L)$ for \emph{all} rank $n$ unimodular lattices $L$.  We thus show the following theorem, which generalises our upper bound for $\massI{K}{n}$ proven in Theorem~\ref{thm:massInbounds}.

\begin{theorem} \label{thm:massupperbound}
    Let $K$ be a totally real degree $d$ field and let $L$ be a unimodular lattice of rank $n \geq 3$.  Then
    \begin{equation*}
        \mass{K}{L} \leq \sigma_n \left( \frac{2^{30} \Gamma(\frac{1}{2}) \Gamma(\frac{2}{2}) \cdots \Gamma(\frac{n}{2})}{\pi^{n(n+1)/4}}   \right)^{[K :\Q]}  \Delta_K^{n(n-1)/4}  
        \prod_{i=1}^{\lfloor (n-1)/2 \rfloor} \zeta_K(2i) ,
    \end{equation*}
    where $\Gamma(s)$ is the usual Gamma function, $\zeta_K(s)$ is the Dedekind zeta function of $K$, and where $\sigma_n$ is as defined in (\ref{eq:sigma}).
\end{theorem}

\begin{proof}
    We recall Siegel's mass formula \cite{SiegelI}, which states that
    \begin{equation} \label{eq:siegelformula}
        \mass{K}{L} = \Delta_K^{n(n-1)/4} \frac{1}{d_{\infty}(L, L)} \prod_{\fp \text{ prime}} \frac{1}{d_{\fp}(L, L)},
    \end{equation}
    where $d_{\infty}(L, L)$ is the real local density, $d_{\fp}(L, L)$ is the local density at the prime $\fp$, and where the product runs over all primes $\fp$ of $K$.

    From Siegel \cite{SiegelIII}, we note that the real local density $d_{\infty}(L, L)$ can be explicitly given by 
    \begin{equation*}
        d_{\infty}(L, L) = \Big(  \frac{\pi^{n(n+1)/4}}{2 \Gamma(\frac{1}{2}) \Gamma(\frac{2}{2}) \cdots \Gamma(\frac{n}{2})} \Big)^d ,
    \end{equation*}
    and the local density $d_{\fp}(L, L)$ for all odd primes $\fp$ can be given by (\cite[Hilfssatz~56]{SiegelIII})
    \begin{equation*}
        d_{\fp}(L, L) = \begin{cases}
            \displaystyle{\prod_{i=1}^{\floor{(n-1)/2}} (1 - \norm{\fp}^{-2i} )} & \text{if $n$ odd}, \\
            \displaystyle{ \left( 1 -  \Big( \frac{-1}{\fp} \Big)^{n/2} \norm{\fp}^{-n/2} \right)  \cdot \prod_{i=1}^{\floor{(n-1)/2}} (1 - \norm{\fp}^{-2i} )} & \text{if $n$ even}. \\
        \end{cases}
    \end{equation*}

    It therefore remains to obtain lower bounds for $d_{\fp}(L, L)$ in the case of dyadic primes $\fp$. 
    
    Let $\fp$ be a dyadic prime and consider $L$ over the local field $K_{\fp}$.  We recall from Pfeuffer \cite[p.~382]{Pfeuffer71} that we can decompose $L$ as 
    \begin{equation} \label{eq:Ldecomp}
        L = N \perp G,
    \end{equation}
    where $N$ is an unramified lattice which can be taken to be some orthogonal sum of a finite number of lattices of type $H_1 := \begin{pmatrix} 0 & 1 \\ 1 & 0 \end{pmatrix}$, and $G$ is a unimodular lattice of dimension $m \leq 4$.

    By using the above decomposition, we note from Pfeuffer \cite[p.~385]{Pfeuffer71} that we have the following decomposition of the local dyadic density $d_{\fp}(L, L)$:
    \begin{equation*}
        d_{\fp}(L, L) = d_{\fp}(L, N) \cdot d_{\fp} (G, G).
    \end{equation*}
    Pfeuffer \cite[p.~386]{Pfeuffer71} computed the local density $d_{\fp}(L, N)$ explicitly for all possible cases of the decomposition in (\ref{eq:Ldecomp}). For each case, we can easily verify the following lower bound for $d_{\fp}(L, N)$:
    \begin{equation*}
        d_{\fp}(L, N) \geq \frac{1}{2^{10}} \prod_{i=1}^{\floor{(n-1)/2}} (1 - \norm{\fp}^{-2i}).
    \end{equation*}

    To estimate a lower bound for $d_{\fp}(G, G)$, we can simply give a (rather crude) lower bound using the definition of $d_{\fp}(G, G)$ given by Siegel:
    \begin{equation} \label{eq:dpGG}
        d_{\fp}(G, G) := \frac{A_{\fp^\nu}(G, G)}{\norm{\fp}^{\nu m(m-1)/2}} ,
    \end{equation}
    where $A_{\fp^\nu}(G, G)$ is the number of distinct isometries $\sigma : G \to G$ modulo $\fp^{\nu}$, $m$ is the dimension of $G$, and where $\nu$ is chosen to be sufficiently large.  In particular, by \cite[Hilfssatz~13, p.~542]{SiegelI}, we can choose $\nu$ satisfying $\fp^{\nu} \subset 4 (\fd G)^2$ where $\fd G$ is the discriminant ideal of $G$.  As $G$ is unimodular, this implies we can choose any $\nu > 2 e_{\fp}$ in the definition given in (\ref{eq:dpGG}).

    Thus, by noting the trivial lower bound of $A_{\fp^\nu}(G, G) \geq 1$ for any $\nu \geq 1$, we have
    \begin{equation*}
        d_{\fp}(G, G) = \frac{A_{\fp^{2 e_{\fp} + 1}}(G, G)}{\norm{\fp}^{ (2 e_{\fp} + 1) m(m-1)/2}} \geq \frac{1}{2^{6 f_{\fp} (2 e_{\fp} + 1) }}.
    \end{equation*}

    Therefore, we have the following lower bound for $d_{\fp}(L, L)$:
    \begin{equation} \label{eq:dplowerbound}
        d_{\fp}(L, L)  = d_{\fp}(L, N) \cdot d_{\fp} (G, G) \geq \frac{1}{2^{10}} \cdot \frac{1}{2^{6 f_{\fp} (2 e_{\fp} + 1) }} \prod_{i=1}^{\floor{(n-1)/2}} (1 - \norm{\fp}^{-2i}).
    \end{equation}

    For each dyadic prime $\fp$, let $\xi_L(n, \fp)$ be defined as
    \begin{equation*}
        \xi_L(n, \fp) := \frac{1}{d_{\fp}(L, L)} \prod_{i=1}^{\floor{(n-1)/2}} (1 - \norm{\fp}^{-2i}).
    \end{equation*}

    By thus applying Siegel's mass formula (\ref{eq:siegelformula}), we obtain the following expression for $\mass{K}{L}$ in terms of the factors $\xi_L(n, \fp)$ (essentially generalising Korner's theorem):
    \begin{equation*}
        \mass{K}{L} = \sigma_n \left( \frac{2 \Gamma(\frac{1}{2}) \Gamma(\frac{2}{2}) \cdots \Gamma(\frac{n}{2})}{\pi^{n(n+1)/4}}   \right)^{[K :\Q]}  \Delta_K^{n(n-1)/4}  
    \prod_{\fp | 2} \xi_L(n, \fp) 
    \prod_{i=1}^{\lfloor (n-1)/2 \rfloor} \zeta_K(2i) ,
    \end{equation*}
    where $\sigma_n$ is defined as in (\ref{eq:sigma}).  Whilst our methods cannot evaluate $\xi_L(n, \fp)$ explicitly, our bounds for $d_{\fp}(L, L)$ in (\ref{eq:dplowerbound}) imply the upper bound 
    $$\xi_L(n, \fp) \leq 2^{6 f_{\fp} (2 e_{\fp} + 1) + 10}$$
    for each dyadic prime $\fp$.  
    Therefore, as $\sum_{\fp | 2} e_{\fp} f_{\fp} = d$, this gives us the upper bound
    \begin{equation*}
        \prod_{\fp | 2} \xi_L(n, \fp) \leq \prod_{\fp | 2} 2^{6 f_{\fp} (2 e_{\fp} + 1) + 10} \leq  2^{30d},
    \end{equation*}
    which proves the theorem.
\end{proof}

By again taking logarithms of $\mass{K}{L}$ in the above theorem, we obtain the following asymptotic upper bound for $\mass{K}{L}$:
\begin{align} \label{eq:logmassasymp2}
    \log \mass{K}{L} &\leq \frac{d}{4} n^2 \log{n} + n^2 \bigg( \nquadcoeff \bigg) - \frac{d}{4} n \log{n} \\ %
    &\quad + n \bigg(  \frac{d \log{(8 \pi )} + d - \log \Delta_K}{4} \bigg) + \bigO{ \log{n} }. \nonumber
\end{align}

\section{Bounds on the number of rank $n$ unimodular $\co$-lattices} \label{sec:unimodular}

In this section, we shall prove bounds on the number of rank $n$ unimodular $\co$-lattices, which we denote as $\UL{K}{n}$.  At the heart of bounding $\UL{K}{n}$, this requires us to prove some known bounds on the class number $\cl_K(L)$ for all rank $n$ unimodular $\co$-lattices $L$.  Recall from the definition of the mass constant (\ref{eq:mass}), we have the trivial bounds
\begin{equation*}
    2 \mass{K}{L} \leq \cl_K(L) \leq \mass{K}{L} \cdot \max_{\Lambda \in \text{gen}(L)}  \card{ \Aut(\Lambda) }.
\end{equation*}
Thus, for an explicit bound on the class number $\cl_K(L)$, it suffices to obtain bounds on the mass constant $\mass{K}{L}$ and bounds on the maximal order of $\Aut(\Lambda)$ for an arbitrary rank $n$ unimodular $\co$-lattice $\Lambda$.  

\subsection{Bounding the orders of automorphism groups}
Let $L$ be a rank $n$ unimodular $\co$-lattice. By choosing an $\co$-basis for $L$, we can consider $\Aut(L)$ as a finite subgroup of $\GL_n(\co)$.  Thus, we can bound $ \card{ \Aut(L) }$ by using known upper bounds on the maximal order of a finite subgroup of $\GL_n(\co)$.

\begin{theorem} \label{thm:autbound}
    Let $K$ be a totally real field of degree $d$, and let $L$ be a rank $n$ unimodular $\co$-lattice.   Then $ \card{ \Aut(L) } \leq 6^{dn/2} (n+1)!$ for all $n > 71$.
\end{theorem}

\begin{proof}
    Let $\Gamma = \Aut(L)$ be the automorphism group of $L$, considered as a finite subgroup of $\GL_n(\mathcal{O}_K)$.  By a theorem of Collins \cite{Collins}, $\Gamma$ contains an abelian normal subgroup $\Delta$ whose index $[\Gamma : \Delta]$ is bounded above by $\leq (n+1)!$ if $n > 71$.

    We can embed $\GL_n(\mathcal{O}_K)$ into $\GL_{dn}(\Z)$ via a restriction of scalars; e.g. see \cite[p.~49]{PlatonovRapinchuk}.  Thus, $\Delta$ maps to a finite abelian subgroup of $\GL_{dn}(\Z)$.  By Friedland \cite{Friedland}, the maximal order of a finite abelian subgroup of $\GL_{dn}(\Z)$ is bounded by $6^{nd/2}$.
    This therefore gives the following upper bound for $\card{ \Aut(L) }$:
    \begin{equation*}
        \card{ \Aut(L) } = \card{ \Gamma } = \card{ \Delta } \cdot [\Gamma : \Delta] \leq 6^{nd/2} (n+1)!.\qedhere
    \end{equation*} %
\end{proof}

\textbf{Remark.}  We note that the finite subgroups of maximal order in $\GL_n(\Z)$ have been fully classified by Feit \cite{Feit}.  In particular, he showed that for $n > 10$, the maximal order of a finite subgroup of $\GL_n(\Z)$ is $2^n n!$, using some unpublished results of Weisfeiler \cite{Weisfeiler} on the structure of finite linear groups.\footnote{Weisfeiler announced his results on the structure of finite linear groups in August 1984, however he disappeared during a hiking trip in Chile soon afterwards.  The story of Weisfeiler's disappearance is an unfortunately sad one and is yet to be resolved; see \url{https://boris.weisfeiler.com/}.}

\subsection{Schur's bounds} Alternatively, we can derive bounds on the maximal order of finite subgroups of $\GL_n(\co)$ by using the following classical result by Schur \cite{Schur} to bound such subgroups (see Serre's note \cite[Section~2]{Serre} or Guralnick--Lorenz \cite[Section~5.3]{GuralnickLorenz} for an English translation of Schur's theorem)

\begin{theorem}[Schur {\cite{Schur}}] \label{thm:schur}
    Let $K$ be a number field and $\ell$ be a prime.  If $\ell > 2$, then let $t_{\ell} := [K(\zeta_\ell) : K]$ and let $m_{\ell}$ be the maximal integer $a$ such that $K(\zeta_\ell)$ contains $\zeta_{\ell^a}$.  Define
    \begin{equation*}
        M_K(n, \ell) := m_{\ell} \cdot \floor{ \frac{n}{t_{\ell}}} + \floor{ \frac{n}{\ell t_{\ell}}} + \floor{ \frac{n}{\ell^2 t_{\ell}}} + \dots .
    \end{equation*}
    If $\ell = 2$, then let $t_2 = [K(i) : K]$ and let $m_2$ be the maximal integer $a$ such that $K(i)$ contains $\zeta_{2^a}$.  Define
    \begin{equation*}
        M_K(n, 2) = n + m_2 \cdot \floor{ \frac{n}{t_2}} + \floor{ \frac{n}{2 t_2}} + \floor{ \frac{n}{4 t_2}} + \dots .
    \end{equation*}
    Let $G$ be a finite $\ell$-subgroup of $\GL_n(\C)$ such that $\Tr(g) \in K$ for all $g \in G$. Then $v_\ell(G) \leq M_K(n, \ell)$.
\end{theorem}

Using Schur's bound, we can derive the following asymptotic upper bound for the maximal order of a finite subgroup of $\GL_n(\co)$.

\begin{theorem} \label{thm:schurautbound}
    Let $K$ be a totally real field.  Let $M_n$ be the maximal order of a finite subgroup of $\GL_n(\co)$. Then $\log M_n \leq n \log n + \bigO{n}$ as $n \to \infty$.  In particular, if $K$ is a real quadratic field, then
    \begin{equation*}
        \log M_n \leq %
        \begin{cases}
            n \log n + (\mathcal{A} + \log(2)/2 ) n + \bigO{ \sqrt{n} \log n }  & \text{if } K = \Q(\sqrt{2}), \\[5pt]
            n \log n + \Big(\mathcal{A} + \frac{p \log p }{ (p-1)^2 } \Big) n  + \bigO{ \sqrt{n} \log{n} }  & \text{if } K = \Q(\sqrt{p}), p \equiv 1 \text{ mod } 4, \\[5pt]
            n \log n + \mathcal{A} n + \bigO{ \sqrt{n} \log{n} }  & \text{otherwise},
        \end{cases}
    \end{equation*}
    as $n \to \infty$, where $\mathcal{A}$ is the explicit real constant:
    \begin{equation*}
        \mathcal{A} := \frac{\log{2}}{2}  - 1 + \sum_{\ell \text{ prime}} \frac{\log (\ell)}{(\ell-1)^2} = 0.57352 \dots.
    \end{equation*}
\end{theorem}

\begin{proof}
    We fix some totally real field $K$ of degree $d$ and discriminant $\Delta_K$, and we assume that $n$ is sufficiently large; in particular we will assume that $n > \Delta_K^2$ and $n > d^2$.   Let $\mathcal{F}_n := \left( n + \floor{\sqrt{n}} \right)!$.  We first recall Legendre's formula which states
    \begin{equation} \label{eq:legendre}
        v_\ell(\mathcal{F}_n) = \floor{\frac{n + \floor{\sqrt{n}}}{\ell}} + \floor{\frac{n + \floor{\sqrt{n}}}{\ell^2}} + \floor{\frac{n + \floor{\sqrt{n}}}{\ell^3}} + \dots.
    \end{equation}
    Also note that, by Stirling's approximation, we have
    \begin{align*}
        \log \mathcal{F}_n &=  \left( n + \floor{\sqrt{n}} \right) \log \left( n + \floor{\sqrt{n}} \right) - \left( n + \floor{\sqrt{n}} \right) + O( \log n) \\
         &= n \log {n} - n + \bigO{ \sqrt{n} \log n}.
    \end{align*}
    
    We first show that $v_\ell( \mathcal{F}_n ) \geq M_K(n, \ell)$ for all primes $\ell > \sqrt{n}$.  Indeed, if $\ell > \sqrt{n}$, then $\ell$ is unramified in $K$, and $t_\ell = [K(\zeta_{\ell}) : K] = \ell - 1$ and $m_{\ell} = 1$, thus
    \begin{align*}
        M_K(n, \ell) &= m_{\ell} \cdot \floor{ \frac{n}{t_{\ell}}} + 
        \sum_{i=1}^\infty \floor{ \frac{n}{\ell^i t_\ell}}
        = \sum_{i=0}^\infty \floor{ \frac{n}{\ell^i (\ell-1)}}
        = 
        \sum_{i=1}^\infty \floor{ \frac{n (1 + \frac{1}{\ell-1})}{\ell^i}}
        \leq v_\ell ( \mathcal{F}_n ),
    \end{align*}
    where the last inequality follows from (\ref{eq:legendre}).

    Now let's consider the case of odd primes $\ell < \sqrt{n}$ such that $[K(\zeta_{\ell}) : K] = \ell - 1$ and $\ell > d$. Again we have $t_{\ell} = \ell - 1$ and $m_{\ell} = 1$.  Thus, we have the following:
    \begin{align*}
        M_K(n, \ell) - v_\ell ( \mathcal{F}_\ell ) 
        &\leq 
        \sum_{i=0}^\infty \floor{ \frac{n}{\ell^i (\ell-1)}} - \sum_{i=1}^\infty \floor{ \frac{n}{\ell^i }}
        \leq 
        \sum_{i=0}^\infty  \frac{n}{\ell^i (\ell-1)} - \sum_{i=1}^\infty  \frac{n}{\ell^i } 
        + \log_\ell(n) 
        \\
        &\leq n \Big( \frac{1/(\ell-1)}{1 - 1/\ell} - \frac{1/\ell}{1 - 1/\ell} \Big) + \log_\ell(n) \\
        &= \frac{n}{(\ell-1)^2} + \log_\ell(n).
    \end{align*}

    Now, let us consider the remaining small odd primes, where either $\ell < d$ or $[K(\zeta_{\ell}) : K] < \ell - 1$.  By tower law, we have $[K(\zeta_\ell) : K] \cdot [ K : \Q] = [K(\zeta_\ell) : \Q] \geq [\Q(\zeta_\ell) : \Q] = \ell - 1$.  This gives the lower bound $t_{\ell} = [K(\zeta_\ell) : K] \geq (\ell-1)/d$. Similarly $m_\ell \leq \floor{ \log_\ell(d) } + 1$. Thus
    \begin{align*}
        M_K(n, \ell) - v_\ell ( \mathcal{F}_n ) 
        &\leq 
        \floor{ \log_{\ell}(d) } \floor{ \frac{dn}{(\ell-1)}} + \sum_{i=0}^\infty \floor{ \frac{dn}{\ell^i (\ell-1)}} - \sum_{i=1}^\infty \floor{ \frac{n}{\ell^i}}
        \\
        & \leq \frac{\floor{ \log_{\ell}(d) } dn}{(\ell-1)} +  
        \sum_{i=0}^\infty \frac{dn}{\ell^i (\ell-1)} - \sum_{i=1}^\infty  \frac{n}{\ell^i}
        + \log_\ell(n) \\
        &\leq \frac{\floor{ \log_{\ell}(d) } dn}{(\ell-1)} + \frac{n ((d-1) \ell + 1)}{(\ell-1)^2} + \log_\ell(n).
    \end{align*}
    Finally, for the prime $\ell = 2$, as $K$ is totally real, we have $t_{2} = 2$ and $m_2 \leq \floor{\log_2(d)} + 2$.  In particular
    \begin{align*}
        M_K(n, 2)  \leq n + (\floor{\log_2(d)} + 1) \floor{ \frac{n}{2}} + v_2 (\mathcal{F}_n ).
    \end{align*}

    By putting together the contributions from all primes $\ell$, we get the following bound for $\log M_n - \log \mathcal{F}_n$:
    \begin{align*}
        \log M_n  - \log{ \mathcal{F}_n} &\leq \sum_{\ell \text{ prime}} \big( M_K(n, \ell) - v_\ell( \mathcal{F}_n)  \big) \log \ell \\
        &\leq  \frac{\floor{\log_2{d} } + 3}{2} n \log 2 + 
        \sum_{\substack{\ell \text{ odd} \\ \ell < \sqrt{n} }} \bigg( \frac{n}{(\ell-1)^2} + \log_\ell(n) \bigg) \log{\ell} \\
        &\quad + \sum_{\substack{\ell \text{ odd, } \ell < \sqrt{n} \\ [K(\zeta_\ell) : K] < \ell - 1 \text{ or } \ell < d}} n \bigg(  \frac{\floor{ \log_\ell(d) } d}{(\ell-1)} + \frac{ (d-1) \ell }{(\ell-1)^2}   \bigg) \log \ell.
    \end{align*}
    Therefore, recalling that $\log \mathcal{F}_n = n \log n - n + O(\sqrt{n} \log n)$, we get the following asymptotic upper bound for $\log M_n$:
    \begin{align*}
        \log M_n &\leq n \log n + \bigg(  \Big( \frac{\floor{\log_2{d}} + 1}{2} \Big) \log 2 - 1 +   \sum_{\ell \text{ prime}}  \frac{\log \ell}{(\ell-1)^2} \\[5pt]
        &\qquad +  \sum_{\substack{\ell \text{ odd prime} \\ [K(\zeta_\ell) : K] < \ell-1}}  \bigg( \frac{\floor{ \log_{\ell}(d) } d}{(\ell-1)} + \frac{ (d-1) \ell}{(\ell-1)^2}  \bigg) \log \ell   \bigg) n + \bigO{ \sqrt{n} \log n }.
    \end{align*}

    Note that the sum $\sum_{\ell \text{ prime}} \log \ell / (\ell-1)^2$ converges and equals $\approx 0.53381 \dots$.  This proves that $\log M_n = n \log n + \bigO{n}$ as $n \to \infty$.

    \bigskip
    Let's now consider the case of real quadratic fields $K$, where $d = 2$.  Then for any odd prime $\ell$, clearly $m_{\ell} = 1$.  This gives us the upper bound:
    \begin{align*}
        \log M_n &\leq n \log n  + \Bigg(  \Big( \frac{m_2 - 1}{2} \Big) \log{2}   -1  +\sum_{\ell \text{ prime}}  \frac{\log \ell}{(\ell-1)^2} + \sum_{\substack{\ell \text{ odd prime} \\ [K(\zeta_\ell) : K] < \ell-1}} \frac{\ell \log \ell}{(\ell-1)^2}   \Bigg) n \\
        &\qquad + \bigO{ \sqrt{n} \log n }.
    \end{align*}
    For $\ell = 2$, note that $m_2 = 3$ if $K = \Q(\sqrt{2})$, otherwise $m_2 = 2$.   
    For an odd prime $\ell$, we recall that the unique quadratic subfield of $\Q(\zeta_{\ell})$ is $\Q(\sqrt{(-1)^{(\ell-1)/2} \ell})$; e.g. see \cite[p.~202]{IrelandRosen}.  Thus, as $K$ is real quadratic, then if $[K (\zeta_\ell) : K] < \ell - 1$, then this implies $\ell \equiv 1$ (mod 4) and $K = \Q(\sqrt{\ell})$.  
    This therefore proves the theorem.
\end{proof}

\begin{theorem} \label{thm:classlowerasymp}
    Let $K$ be a totally real field of degree $d$ and discriminant $\Delta_K$. Then the logarithm of the class number of $I_n$ satisfies the lower bound
    \begin{align*}
        \log \cl_K(I_n) &\geq \frac{d}{4} n^2 \log{n} + n^2 \Big( \nquadcoeff \Big)  - \frac{d}{4} n \log n  \\
        &\quad + n \bigg(  \nlinearcoeff \bigg) + \bigO{ \log{n} } \nonumber
    \end{align*}
    as $n \to \infty$.
\end{theorem}

\begin{proof}
    By the trivial bound $2 \massI{K}{n} \leq \cl_K(I_n)$, we have the bound
    \begin{equation*}
        \log 2 + \log \massI{K}{n} \leq \log \cl_K(I_n). %
    \end{equation*}
    Now using 
    the asymptotic for $\massI{K}{n}$ in Theorem~\ref{thm:massasymp} yields the desired asymptotic lower bound for $\log \cl_K(I_n)$.
\end{proof}

\begin{theorem} \label{thm:ULasymp}
    Let $K$ be a totally real field of degree $d$ and discriminant $\Delta_K$.  Then the number of isometry classes of rank $n$ unimodular lattices $L$, denoted by $\UL{K}{n}$, satisfies the lower bound
    \begin{align} \label{eq:ULlower}
        \log \UL{K}{n} &\geq \frac{d}{4} n^2 \log{n} + n^2 \Big( \nquadcoeff \Big)  - \frac{d}{4} n \log n \\ %
        &\quad + n \bigg(  \nlinearcoeff \bigg) + \bigO{ \log{n} } ,\nonumber
    \end{align}
    and the upper bound
    \begin{align} \label{eq:ULupper}
        \log \UL{K}{n} &\leq \frac{d}{4} n^2 \log{n} + n^2 \Big( \nquadcoeff \Big)  + \Big(1 - \frac{d}{4} \Big) n \log n \\ %
        &\quad + n \bigg(  \nlinearcoeffa + \frac{d \log 6}{2} - 1 \bigg) + \bigO{ \log{n} } \nonumber
    \end{align}
    as $n \to \infty$.
\end{theorem}

\begin{proof}
    Since we clearly have the lower bound $\cl_K(I_n) \leq \UL{K}{n}$, Theorem~\ref{thm:classlowerasymp} gives the lower bound in (\ref{eq:ULlower}).  For the upper bound, let $\SUL{K}{n}$ denote the set of all rank $n$ unimodular $\co$-lattices, and let $\mathcal{G}_n$ denote the set of genera of rank $n$ unimodular $\co$-lattices. Then we note that
    \begin{equation*}
        \UL{K}{n} = \sum_{L \in \mathcal{G}_n} \cl_K(L) \leq \card{ \mathcal{G}_n } \cdot \max_{L \in \SUL{K}{n}} \big( \mass{K}{L} \cdot \card{ \Aut(L) } \big).
    \end{equation*}

    By thus using our upper bounds for $\mass{K}{L}$ from Theorem \ref{thm:massupperbound} and our upper bound for the automorphism group $\Aut(L)$ from Theorem \ref{thm:autbound}, we obtain the following upper bound:
    \begin{equation*}
        \UL{K}{n} \leq \sqrt{5}^d \sigma_n \left( \frac{2^{30} \Gamma(\frac{1}{2}) \Gamma(\frac{2}{2}) \cdots \Gamma(\frac{n}{2})}{\pi^{n(n+1)/4}}   \right)^{d}  \Delta_K^{n(n-1)/4}  
        \prod_{i=1}^{\lfloor (n-1)/2 \rfloor} \zeta_K(2i) \cdot 6^{nd/2} (n+1)!.
    \end{equation*}

    By thus taking logs and using the asymptotic in (\ref{eq:logmassasymp2}), we obtain the upper bound given in (\ref{eq:ULupper}).  This proves the claim.
\end{proof}

\section{Lower bounds for $\MRU{K}{n}$} \label{sec:lower}

In order to relate a lower bound for the number of rank $n$ unimodular $\co$-lattices to a lower bound for the number of rank $n$ \emph{indecomposable} unimodular $\co$-lattices, we aim to prove that $\IUL{K}{n} \sim \UL{K}{n}$ as $n \to \infty$ by using the following theorem of Wright \cite{Wright68}:

\begin{theorem}[Wright {\cite{Wright68}}] \label{thm:wright} Let $\{T_n\}_{n \geq 1}$ and $\{t_n\}_{n\geq 1}$ be two sequences such that $T_n$ is the Euler transform of $t_n$, i.e. where
\begin{equation*}
    1 + \sum_{n=1}^\infty T_n X^n = \prod_{k=1}^\infty (1 - X^k)^{-t_k}.
\end{equation*}
Then if 
\begin{equation} \label{eq:Wrightcondition}
    \sum_{s=1}^{n-1} H_s H_{n-s} = o(H_n)
\end{equation}
as $n \to \infty$, where the sequence $\{H_n\}$ is either $\{T_n\}$ or $\{t_n\}$, then $T_n \sim t_n$ as $n \to \infty$.
\end{theorem}

\subsection{Case $d \geq 3$} By applying Wright's theorem, we can thus prove the following:

\begin{theorem} \label{thm:wrightasymp}
    Let $K$ be a totally real field of degree $d \geq 3$. Then $\IUL{K}{n} \sim \UL{K}{n}$ as $n \to \infty$.  That is, the number of indecomposable rank $n$ unimodular $\co$-lattices is asymptotic to all rank $n$ unimodular $\co$-lattices as $n \to \infty$.
\end{theorem}

\begin{proof}
    As before, we fix some totally real field $K$. 
    It's clear that $\UL{K}{n}$ is the Euler transform of $\IUL{K}{n}$.  We now show that (\ref{eq:Wrightcondition}) holds for the sequence $H_n = \UL{K}{n}$.

    For brevity and to simplify the calculations, we shall put
    \begin{equation*}
        P_K := \nquadcoeff
    \end{equation*}
    and let
    \begin{equation*}
        Q_K(n) := \frac{d}{4} n^2 \log{n} + P_K n^2 - \frac{d}{4} n \log{n} .
    \end{equation*}
    Note that $Q_K(n)$ is an increasing sequence for sufficiently large $n$.  In particular, we have
    \begin{equation*}
        Q_K(n) - Q_K(n-1) 
        = \frac{d}{2} n \log{n} + \bigO{n}.
    \end{equation*}
    
    Using the first three terms on our bounds on $\log \UL{K}{n}$ from Theorem~\ref{thm:ULasymp}, for any $\eps > 0$, we have that
    \begin{equation} \label{eq:classbounds}
        Q_K(n) - \eps n \log n \leq \log \UL{K}{n} \leq Q_K(n) + (1 + \eps) n \log{n}
    \end{equation}
    for all sufficiently large $n$.

    We choose $\eps = \frac{1}{4}$ and we let $B > 0$ be such that (\ref{eq:classbounds}) holds for all $n > B$. Let $M$ be the maximum of $\UL{K}{n}$ for $n = 1, \dots, B$.  Then we note that, for all $n > 2B$ we have
    \begin{align}
        \sum_{k=1}^{n-1} H_k H_{n-k} 
        &= \sum_{k=1}^{B} H_k H_{n-k} + \sum_{k=B+1}^{n-B-1} H_k H_{n-k} + \sum_{k=n-B}^{n-1} H_k H_{n-k} \nonumber \\
        &\leq 2 \sum_{k=1}^B M \exp{ \Big(  Q_K(n-k) + (1+\eps) (n-k) \log (n-k) \Big)} \nonumber \\
        &\quad + \sum_{k=B+1}^{n-B-1}  
        \exp{ \Big( Q_K(k) + (1+\eps) k \log k  +  Q_K(n-k) + (1 + \eps) (n-k) \log(n-k) 
          \Big)  } \nonumber \\
        &\leq 2BM \exp{ \Big( Q_K(n-1)  + (1+ \eps) (n-1) \log{n}  \Big)  } \nonumber \\
        &\quad + n \exp{ \Big( Q_K(n-1)  + (1+ \eps) (n-1) \log{n} + \bigO{1} \Big)  } \nonumber \\
        &\leq \exp{ \Big(  Q_K(n) + (1 + \eps - \frac{d}{2}) n \log{n} + \bigO{n}  \Big) } \nonumber \\
        &=  o \big( \exp{(Q_K(n) - \eps n \log n )} \big) = o(H_n), \label{eq:lastline}
    \end{align}
    where the last inequality holds as $1 + \eps - \frac{d}{2} \leq - \frac{1}{4}$, as $d > 2$.  Thus, we have that (\ref{eq:Wrightcondition}) holds for $H_n = \UL{K}{n}$, and so by Theorem~\ref{thm:wright}, this proves that $\UL{K}{n} \sim \IUL{K}{n}$ as $n \to \infty$.
    
\end{proof}

As a corollary, using Theorem~\ref{thm:ULasymp} this gives the asymptotic for $\log \IUL{K}{n}$:
\begin{equation*}
    \log \IUL{K}{n} = \frac{d}{4} n^2 \log{n} + n^2 \Big( \nquadcoeff \Big)  +  \bigO{n \log n} %
\end{equation*}
as $n \to \infty$.

\subsection{Case $d = 2$}

Note that line (\ref{eq:lastline}) in our proof of Theorem~\ref{thm:wrightasymp} above required that $d > 2$.  In the case of real quadratic fields, where $d = 2$, we must redo the computations, keeping track of further error terms for $\UL{K}{n}$.  This will allow us to apply Wright's criterion for all but finitely many real quadratic fields $K$.

\begin{theorem} \label{thm:wrightasymp2}
    Let $K$ be a real quadratic field with discriminant $\Delta_K$.  Let $\Autn$ be a constant such that $\log \card{ \Aut(L) } \leq n \log n + \Autn n + \bigO{\sqrt{n} \log n }$ for any rank $n$ unimodular $\co$-lattice.  Then if 
    \begin{equation} \label{eq:wright_deg2_condition}
        \frac{\log \Delta_K - 2 \log(2 \pi) - 2}{2}  > \Autn + \sum_{\fp | 2} \floor{\frac{e_{\fp}}{2}} \log \norm{\fp}
    \end{equation}
    holds, then $\UL{K}{n} \sim \IUL{K}{n}$ as $n \to \infty$.
\end{theorem}

\begin{proof}
    We proceed in a similar way to the proof of Theorem~\ref{thm:wrightasymp}, but must now furthermore keep track of our precise bounds for the $\bigO{n}$ error terms for $\UL{K}{n}$.  This requires using our best known bounds for the $\bigO{n}$ error term for $\log \card{ \Aut(\Lambda) }$.

    As before, for brevity and to simplify the calculations, we shall put
    \begin{equation*}
        P_K := \nquadcoeff \quad \text{and} \quad R_K := \nlinearcoeffa
    \end{equation*}
    and let
    \begin{equation*}
        Q_K(n) = \frac{d}{4} n^2 \log{n} + P_K n^2 + \Big(1 - \frac{d}{4} \Big) n \log{n} + (R_K + \Autn) n.
    \end{equation*}
    As before, we have that $Q_K(n)$ is an increasing sequence for sufficiently large $n$, and in particular we have
    \begin{equation*}
        Q_K(n) - Q_K(n-1) = \frac{d}{2} n \log{n} + \Big(2 P_K + \frac{d}{4}  \Big) n  + \Big(1 - \frac{d}{4} \Big) \log{n} + \bigO{1}.
    \end{equation*} 
    
    Note from (\ref{eq:ULupper}) that we have $\log \UL{K}{n} \leq Q_K(n) + \bigO{\log n}$ as $n \to \infty$.   Recall from Theorem~\ref{thm:ULasymp}  that there exists some constant $C > 0$, such that we have 
    \begin{equation} \label{eq:classboundsquad}
        Q_K(n) - n \log n - \Big( \Autn + \sum_{\fp | 2} \floor{\frac{e_{\fp}}{2}} \log \norm{\fp} \Big) - C \sqrt{n} \log{n} \leq \log \UL{K}{n} \leq Q_K(n) +  C \sqrt{n} \log n 
    \end{equation}
    for all sufficiently large $n$.

    Let $B > 0$ be such that (\ref{eq:classboundsquad}) holds for all $n > B$, and let $M$ be the maximum of $\UL{K}{n}$ for $n = 1, \dots, B$.  Then we note that, for all $n > 2B$ we have
    \begin{align*}
        \sum_{k=1}^{n-1} H_k H_{n-k} 
        &= \sum_{k=1}^{B} H_k H_{n-k} + \sum_{k=B+1}^{n-B-1} H_k H_{n-k} + \sum_{k=n-B}^{n-1} H_k H_{n-k} \nonumber \\
        &\leq 2 \sum_{k=1}^B M \exp{ \Big(  Q_K(n-k)  + C \sqrt{n-k} \log (n-k) \Big)} \\
        &\quad + \sum_{k=B+1}^{n-B-1}  
        \exp{ \Big( Q_K(k) + C \sqrt{k} \log k +  Q_K(n-k) +  C \sqrt{n-k} \log(n-k)
          \Big)  } \\
        &\leq 2BM \exp{ \Big( Q_K(n-1)  + C \sqrt{n} \log n  \Big)  } \\
        &\quad + n \exp{ \Big( Q_K(n-1) + C \sqrt{n} \log n + \bigO{1} \Big)  } \\
        &\leq \exp{ \Big(  Q_K(n) - n \log n - (2P_K + 1/2) n + \bigO{\log n}  \Big) } \\
        &=  o \big( \exp{(Q_K(n) -  n \log n - (\Autn + \sum_{\fp | 2} \floor{e_{\fp}/2} \log \norm{\fp}) n + \bigO{\log n} )} \big) = o(H_n),
    \end{align*}
    where the last inequality holds from our assumption of (\ref{eq:wright_deg2_condition}). Thus, as before we have that (\ref{eq:Wrightcondition}) holds for $H_n = \UL{K}{n}$, and so by Theorem~\ref{thm:wright}, this proves that $\UL{K}{n} \sim \IUL{K}{n}$ as $n \to \infty$.
\end{proof}

\begin{cor} \label{cor:d2}
    Let $K = \Q(\sqrt{D})$ be a real quadratic field such that $D$ is squarefree and $D \in \{919$, $921$, $922$, $923$, $926\}$ or $D \geq 930$.  Then $\IUL{K}{n} \sim \UL{K}{n}$ as $n \to \infty$.  In particular, Theorem~\ref{thm:wrightasymp2} holds for all but at most $\numquadexceptions$ exceptional real quadratic fields $K$.
\end{cor}

\begin{proof}
    Note that, by Theorem~\ref{thm:autbound}, we can take $\Autn = \log{6} - 1$ for all quadratic fields $K$.   Furthermore, we clearly have $\sum_{\fp | 2} \floor{\frac{e_{\fp}}{2}} \log \norm{\fp} \leq \log {2}$.

    Thus, if $\Delta_K > 576 \pi^2 \approx 5684.89\dots$, then a standard calculation gives
    \begin{equation*}
        \frac{\log \Delta_K - 2 \log(2 \pi) - 2}{2} > \log 6 - 1 + \log 2,
    \end{equation*}
    and thus the condition (\ref{eq:wright_deg2_condition}) holds for all real quadratic fields of discriminant $\Delta_K > 576 \pi^2$.

    Therefore, it remains to check condition (\ref{eq:wright_deg2_condition}) for all the real quadratic fields $K$ of discriminant at most $5684$. Using Theorem~\ref{thm:autbound}, we can take $\Autn$ to be either $\log{6} - 1$, or, by using Theorem~\ref{thm:schurautbound}, we can take 
    \begin{equation*}
        \Autn = \begin{cases}
            \mathcal{A} + \log{2}/2  & \text{if } K = \Q(\sqrt{2}) \\
            \mathcal{A} + \frac{p \log p }{(p-1)^2} & \text{if } K = \Q(\sqrt{p}), p \equiv 1 \pmod{4} \\
            \mathcal{A} & \text{otherwise}, 
        \end{cases}
    \end{equation*}
    whichever is smaller.  By checking (\ref{eq:wright_deg2_condition}) with SageMath for each of these fields, we verify that condition (\ref{eq:wright_deg2_condition}) holds exactly for all squarefree $D$ such that $D \in \{919, 921, 922, 923, 926\}$ or $D \geq 930$.  

    Counting the number of squarefree $D < 930$, not including $919, 921, 922, 923, 926$, gives exactly $\numquadexceptions$ real quadratic fields for which (\ref{eq:wright_deg2_condition}) does not hold.
\end{proof}

\subsection{Case $K = \Q$}

We note that our proofs of Theorem~\ref{thm:wrightasymp} and Theorem~\ref{thm:wrightasymp2} do not work in the case $d = 1$.  However, in the particular case where $K = \Q$, we can instead use some known results of Bannai \cite{Bannai}.

\begin{theorem}[Bannai {\cite[Theorem~2]{Bannai}}] \label{thm:bannai}
    Let $L$ be a rank $n$ positive definite unimodular $\Z$-lattice. Let $\mass{\Q}{L}$ be the usual Siegel mass of $L$ as given in (\ref{eq:mass}), and let $\mathrm{mass}'_{\Q}(L)$ be the mass of all classes in $\text{gen}(L)$ with nontrivial automorphisms, i.e.:
    \begin{equation} \label{eq:nontrivmass}
        \mathrm{mass}'_{\Q}(L) := \sum_{\substack{  \Lambda \in \gen(L) \\ \card{\Aut(\Lambda)} > 2 }} \frac{1}{\card{\Aut(\Lambda)}}.
    \end{equation}
    Then
    \begin{equation*}
        \frac{\mathrm{mass}'_{\Q}(L)}{ \mass{\Q}{L}} \leq \begin{cases}
            30 (\sqrt{2 \pi })^n / \Gamma(\frac{n}{2}) & \text{if $L$ odd and $n \geq 43$}, \\
            2^{n+1} (\sqrt{2 \pi} )^n / \Gamma( \frac{n}{2}) & \text{if $L$ even and $n \geq 144$}.
        \end{cases}
    \end{equation*}
    In particular, the ratio $ \mathrm{mass}'_{\Q}(L)/\mass{\Q}{L}$ rapidly tends to $0$ as $n \to \infty$.
\end{theorem}

\begin{theorem} \label{thm:Q}
    We have
    \begin{equation*}
        \log \IUL{\Q}{n} \geq \frac{1}{4} n^2 \log{n} - n \Big( \frac{2 \log ( 2 \pi ) + 3}{8} \Big)  + \bigO{n \log n }
    \end{equation*}
    as $n \to \infty$.
\end{theorem}

\begin{proof}
    We shall apply Theorem~\ref{thm:bannai} to each genus of unimodular lattices $L$.  In particular, this implies that, for any $\eps > 0$, we have
    \begin{equation*}
        \sum_{\substack{  \Lambda \in \gen(L) \\ \card{\Aut(\Lambda)} = 2 }} \frac{1}{\card{ \Aut(\Lambda) }} \geq (1 - \eps)  \mass{\Q}{L}
    \end{equation*}
    for all sufficiently large $n$. Let $\mathcal{G}_n$ be the finite set of genera of rank $n$ unimodular lattices. Given that any lattice with a trivial automorphism group must necessarily be indecomposable, we have that, for any $\eps > 0$,
    \begin{equation*}
        \frac{1}{2} \IUL{\Q}{n} \geq  \sum_{L \in \mathcal{G}_n} \sum_{\substack{  \Lambda \in \gen(L) \\ \card{\Aut(\Lambda)} = 2 }} \frac{1}{\card{ \Aut(\Lambda)}} \geq \sum_{\substack{  \Lambda \in \gen(I_n) \\ \card{\Aut(\Lambda)} = 2 }} \frac{1}{\card{ \Aut(\Lambda)}} \geq (1 - \eps) \massI{\Q}{n}.
    \end{equation*}
    for all sufficiently large $n$.
    Thus, we have $ \log \IUL{\Q}{n} \geq \log \massI{\Q}{n} + \bigO{1}$.  Using the asymptotic expansion of $\log \massI{\Q}{n}$ given in (\ref{eq:logmassasymp}), this proves the theorem.
\end{proof}

\noindent \textbf{Remark}. To avoid understating Bannai's work, we remark that our proof of Theorem~\ref{thm:Q} above only requires the result that almost all unimodular $\Z$-lattices of sufficiently large rank are indecomposable.  We emphasize that Bannai proves the far stronger result that almost all unimodular $\Z$-lattices have \emph{trivial automorphism group}.  It seems reasonable to conjecture that an analogue of Bannai's theorem holds for any totally real field $K$, i.e. that $\text{mass}_K'(L_n) / \mass{K}{L_n} \to 0$ as $n \to \infty$, for any sequence of rank $n$ unimodular $\co$-lattices $L_n$.

\section{Upper bounds for $\MRU{K}{n}$} \label{sec:upper}

\begin{theorem} \label{thm:MRUupper}
    Let $K$ be any totally real field.  Then
    \begin{equation*}
        \log \MRU{K}{n} \leq \frac{d}{4} n^2 \log{n} + n^2 \Big( \nquadcoeff \Big)  + \bigO{n \log n} 
    \end{equation*}
    as $n \to \infty$.
\end{theorem}

\begin{proof}
Recall from Lemma~\ref{bounds} that we have the upper bound:
\begin{equation*}
    \MRU{K}{n} \leq \sum_{k=1}^{n+3} \floor{\frac{n+3}{k}} k \IUL{K}{k}.
\end{equation*}
Now, as $\IUL{K}{n} \leq \UL{K}{n}$ and $\UL{K}{n}$ is a non-decreasing sequence, this gives us the bound $\MRU{K}{n} \leq (n+3)^2 \UL{K}{n+3}$.

By thus applying the upper bound for $\UL{K}{n+3}$ given in Lemma~\ref{thm:ULasymp}, we can derive the following explicit upper bounds for $\MRU{K}{n}$:
\begin{align*}
    \log \MRU{K}{n} &\leq \log \UL{K}{n} + 2 \log(n+3)  %
    \\
    &= \frac{d}{4} (n+3)^2 \log(n+3) + (n+3)^2 \Big( \nquadcoeff \Big)  +  \bigO{n \log n} \\
    &= \frac{d}{4} n^2 \log{n} + n^2 \Big( \nquadcoeff \Big)  + \bigO{n \log n} ,
\end{align*}
which proves the theorem.
\end{proof}

\subsection{Proof of Theorem~\ref{thm:mainasymp}} 
We first note that the upper bound claimed in (\ref{eq:main}) immediately follows from Theorem~\ref{thm:MRUupper}, for any totally real field $K$.

To prove the claimed lower bound for $\MRU{K}{n}$, we use the bounds in Lemma~\ref{bounds} to obtain that $\MRU{K}{n} \geq \IUL{K}{n} $.  Here, we consider three cases based on the degree $d$ on $K$.  If $d \geq 3$, then Theorem~\ref{thm:wrightasymp} implies that $\IUL{K}{n} \sim \UL{K}{n}$. Thus we have
\begin{equation*}
    \log \MRU{K}{n} \geq \log \IUL{K}{n} \geq \log \UL{K}{n} + \bigO{1},
\end{equation*}
and then the lower bound for $\UL{K}{n}$ given in Theorem~\ref{thm:ULasymp} implies the claimed lower bound for $\MRU{K}{n}$ in (\ref{eq:main}).  Similarly, for the case $d = 2$, if $K = \Q(\sqrt{D})$ where $D \geq 930$ or $D \in \{919, 921, 922, 923, 926\}$, then Corollary~\ref{cor:d2} implies that $\IUL{K}{n} \sim \UL{K}{n}$, and thus the same argument gives the claimed lower bound for $\MRU{K}{n}$.  Finally, in the case $K = \Q$, the bound in Theorem~\ref{thm:Q} gives the claimed lower bound for $\MRU{\Q}{n}$. \qed
\bigskip

To furthermore prove the claim in (\ref{eq:limsup}), we note that as $\UL{K}{n}$ grows super-exponentially, the power series $f(x) = \sum_{n=1}^\infty \UL{K}{n} X^n$ has a radius of convergence zero.  Thus, by a theorem of Bell \cite[Theorem~1, p.~1]{Bell}, this implies 
$$\limsup_{n \to \infty} \frac{\IUL{K}{n} }{ \UL{K}{n} } = 1,$$
which thus proves the claim in (\ref{eq:limsup}).  In particular, this implies that $\log \IUL{K}{n} = \log \UL{K}{n} + O(1)$ for infinitely many positive integers $n \geq 1$, and thus $\ref{thm:mainasymp}$ holds for infinitely many positive integers $n \geq 1$.

\section{Effective Density results} \label{sec:effective_density}

\begin{theorem} \label{thm:eulerbounds}
    Let $\{a_n\}$ be a sequence of nonnegative integers, and let $\{b_n\}$ be its Euler transform.  Then, for all $n \geq 1$, $\max \{ a_1, \dots, a_n \} \geq (b_n/(n!)^2)^{1/n}$.
\end{theorem}

\begin{proof}
    We recall the following well-known recursive formula between any given sequence $\{a_n\}$ and its Euler transform $\{b_n\}$ (e.g. see \cite[p.~20]{SloanePlouffe}):
    \begin{equation*}
        b_n = \frac{1}{n} \left( c_n + \sum_{k=1}^{n-1} c_k b_{n-k}  \right), \quad \text{ where } \quad   c_n = \sum_{d | n} d a_d.
    \end{equation*}
    Fix some $n \geq 1$, and let $C \geq 1$ be a constant such that $a_k \leq C$ for all $k \leq n$.  We shall prove by induction that $b_k \leq C^k (k!)^{2}$ for all $k \leq n$.

    Firstly, we have that $b_1 = a_1 \leq C$, thus case $k = 1$ clearly holds. Now assume $b_k \leq C^k (k!)^2$ for all $k \leq m$ for some $m$. 
    Note that 
    \begin{equation*}
        c_k \leq \sum_{d | k} d C = C \sigma_1(k) \leq C k^{2}.
    \end{equation*}
    for all $k \leq m$. Thus, we have the upper bound
    \begin{align*}
        b_m &\leq \frac{1}{m} \Big(  C m^{2} + \sum_{k=1}^{m-1} { C k^{2} C^{n-k} ((n-k)!)^2 }  \Big) \\
        &\leq C m + \sum_{k=1}^{m-1} C^m m ((m-1)!)^2 \leq C^m (m!)^2  ,
    \end{align*}
    which proves the induction step.  Letting $C = \max( a_1, \dots, a_n)$ thus proves the theorem.  
\end{proof}

We now aim to obtain lower bounds for the values $\IUL{K}{n}$.  For now, we can now apply the bounds obtained in Theorem~\ref{thm:eulerbounds} to derive lower bounds for $\max \{ \IUL{K}{1}, \dots, \IUL{K}{n} \}$.

\begin{prop} \label{prop:maxIULbounds}
    Let $K$ be a totally real field of degree $d$ and discriminant $\Delta_K$.  If $n \geq 3$, then
    \begin{equation*}
        \max \{ \IUL{K}{1}, \dots, \IUL{K}{n}  \} \geq  \frac{\Delta_K^{(n-1)/4}}{(n!)^{2/n} } \Big( \frac{ 2 \Gamma(\frac{1}{2}) \cdots \Gamma(\frac{n}{2}) \pi^{-n(n+1)/4}}{2^{(n+5)/2} \cdot (2 \zeta(n/2) )^{\delta_{2 \mid n}}} \Big)^{d/n}.
    \end{equation*}
    For the case $n = 2$, then
    \begin{equation*}
        \max \{ \IUL{K}{1}, \IUL{K}{2} \}
        \geq \left( \frac{\Delta_K^{1/2}}{2 \pi^d \cdot 2^{5d/2}}
        \cdot \Big( \frac{(d-1)}{e \log  \Delta_K} \Big)^{d-1} \cdot \frac{0.000181}{d \cdot (2d)! \cdot \Delta_K^{1/d}}  \right)^{1/2}.
    \end{equation*}
\end{prop}

\begin{proof}
    Note that we have $\UL{K}{n} \geq \cl_K(I_n) \geq 2 \massI{K}{n}$ for all $n \geq 1$. Thus by Theorem~\ref{thm:eulerbounds} and using the lower bounds for $\massI{K}{n}$ given in Theorem~\ref{thm:massInbounds}, we have that
    \begin{align*}
        \max \{ \IUL{K}{1}, \dots, \IUL{K}{n}  \} &\geq (\UL{K}{n}/(n!)^2)^{1/n} \\
        &\geq (2 \cdot \massI{K}{n}/(n!)^2)^{1/n} \\
        &\geq \left(  \Big(  \frac{F_n \cdot 60.1^{n(n-1)/4}}{8 \zeta(n/2)} \Big)^d \cdot \frac{2 \exp(-254 n(n-1)/4)}{(n!)^2}  \right)^{1/n}.
    \end{align*}

    Note that, thus in the case of $n = 2$, by using the lower bound for $\massI{K}{2}$ given in (\ref{eq:masslowerI2}), we have
    \begin{align*}
        \max ( \IUL{K}{1}, \IUL{K}{2} ) &\geq (\UL{K}{2}/4 )^{1/2} 
        \geq \Big( \frac{\massI{K}{2}}{2} \Big)^{1/2} \\
        &\geq 
        \left( \frac{\Delta_K^{1/2}}{2 \pi^d \cdot 2^{5d/2}}
        \cdot \Big( \frac{(d-1)}{e \log  \Delta_K} \Big)^{d-1} \cdot \frac{0.000181}{d \cdot (2d)! \cdot \Delta_K^{1/d}}  \right)^{1/2}.\qedhere
    \end{align*}

\end{proof}

\subsection{Proof of Theorem~\ref{thm:finite}}

Let $n \geq 3$ and let $C > 0$ be a fixed constant such that $\MRU{K}{n} < C$.  Using the lower bound for $\MRU{K}{n}$ from Lemma~\ref{bounds} and the bounds from Theorem~\ref{thm:eulerbounds}, we have the following lower bound for $\MRU{K}{n}$:
\begin{align*}
    \MRU{K}{n} &\geq \sum_{k=1}^n \floor{ \frac{n}{k}} k \IUL{K}{k} \geq  \max \{ \IUL{K}{1}, \dots, \IUL{K}{n}  \} \\
    &\geq (2 \cdot \massI{K}{n}/(n!)^2)^{1/n}.
\end{align*}
Thus, we have $\massI{K}{n} \leq C^n (n!)^2 / 2$.  By thus using the bounds for the degree $d$ and discriminant $\Delta_K$ given in (\ref{eq:dboundmass}) and (\ref{eq:discboundmass}), this gives the bounds
\begin{align*}
    d &\leq \frac{\log (C^n (n!)^2 \exp ( 127 n(n-1)/2 ) / 2 ) }{  \log ( (F_n \cdot 60.1^{n(n-1)/4} ) / ( 2^{(n+5)/2} \cdot (2 \zeta(n/2) )^{\delta_{2 \mid n}} ) ) } \\
    &= \frac{4 n \log(C) + 8 \log(n!) + 254n(n-1) - 4\log(2)}{
    \log ( ((2 \cdot 60.1^{n(n-1)/4 }  \prod_{i=1}^n \Gamma(i/2) )  ) / ( 2^{(n+5)/2} \cdot \pi^{n(n+1)/4} \cdot (2 \zeta(n/2) )^{\delta_{2 \mid n}} ) )}
\end{align*}
and 
\begin{align*}
    \Delta_K &\leq \Big( \frac{C^n (n!)^2 / 2}{ (F_n / ( 2^{(n+5)/2} \cdot (2 \zeta(n/2) )^{\delta_{2 \mid n}} ) )^d } \Big)^{4 / (n(n-1))} \\
    &= \Big( \frac{C^n (n!)^2 \pi^{dn(n+1)/4} 2^{d(n+5)/2} (2 \zeta(n/2))^{d\delta_{2|n}} }{2 \cdot (2 \prod_{i=1}^n \Gamma(i/2) )^d  }  \Big)^{4 / (n(n-1))},
\end{align*}
which proves the theorem. \qed

\subsection{Proof of Theorem~\ref{thm:finite2}}

In the case where $n = 2$, then Theorem~\ref{thm:eulerbounds} implies the bound
\begin{equation*}
    \massI{K}{2} \leq \cl_K(I_2) / 2 \leq \UL{K}{2}/2 \leq 2 (\max \{ \IUL{K}{1}, \IUL{K}{2} \})^2 \leq 2C^2.
\end{equation*}
Thus, as the class number $\cl_K(I_2)$ is bounded above by $4C^2$, then Theorem~\ref{thm:pfeuffer} by Pfeuffer proves the claimed finiteness result.

If $d \geq 3$, then we can obtain an effective bound on $\Delta_K$ from (\ref{eq:discbound2}).  This gives us
\[
    \pushQED{\qed}
    \Delta_K \leq \max \left( \exp \big( (12(d-1))^2 \big), \; \Big (\frac{ 2 C^2 \cdot (\pi \cdot 2^{5/2})^d \cdot d \cdot (2d)!   }{ 0.000181 \cdot ((d-1)/e)^{d-1} } \Big)^{\frac{12d}{5d-12}} \right).\qedhere
    \popQED
\]

\section{Bounds on $\MRU{K}{n+4} - \MRU{K}{n}$} \label{sec:shortgaps}

Here, we prove our theorem on bounds on the gaps $\MRU{K}{n+4} - \MRU{K}{n}$:

\subsection{Proof of Theorem~\ref{thm:shortgap}}
Let $K$ be a totally real field.  Using the bounds from (\ref{bounds}), we have
\begin{align*}
    \MRU{K}{n+4} &- \MRU{K}{n} \geq \sum_{k=1}^{n+4} \floor{ \frac{n+4}{k} } k \IUL{K}{k} - \sum_{k=1}^{n+3} \floor{ \frac{n+3}{k} } k \IUL{K}{k} \\
    & \geq \IUL{K}{1} + 2 \Big( \floor{ \frac{n+4}{2} } - \floor{ \frac{n+3}{2} } \Big) \IUL{K}{2} + 3 \Big( \floor{ \frac{n+4}{3} } - \floor{ \frac{n+3}{3} } \Big) \IUL{K}{3}
\end{align*}
If $n$ is even, then $\floor{(n+4)/2} > \floor{(n+3)/2}$, and thus we have the lower bound
\begin{align*}
    \MRU{K}{n+4} - \MRU{K}{n} &\geq \IUL{K}{1} + 2 \IUL{K}{2} \\
    &\geq \max( \IUL{K}{1}, \IUL{K}{2}) \\
    &\geq (\massI{K}{2}/2)^{1/2}.
\end{align*}
Thus, by the same bounds above, this proves have the theorem.

If furthermore $n \equiv 2$ mod 6, then $\floor{(n+4)/2} > \floor{(n+3)/2}$ and $\floor{(n+4)/3} > \floor{(n+3)/3}$, and so we have that
\begin{align*}
    \MRU{K}{n+4} - \MRU{K}{n} &\geq \IUL{K}{1} + 2 \IUL{K}{2} + 3 \IUL{K}{3} \\
    &\geq \max( \IUL{K}{1}, \IUL{K}{2}, \IUL{K}{3}) \\
    &\geq (\massI{K}{3}/18)^{1/3},
\end{align*}
which proves our claim.
\qed

\section{Cardinality and Nonuniqueness of criterion sets} \label{sec:criterion}

Recall that an $n$-universal criterion set is a set $\mathcal{S}$ of rank $n$ $\co$-lattices with the property that any lattice $L$ is $n$-universal if and only if $L$ represents all lattices in $\mathcal{S}$.
We first observe in Lemma~\ref{lem:criterion common} that lattices of the certain form must be contained in all $n$-universal criterion sets.
In Proposition~\ref{pro:criterion replacement}, we show that one may replace lattices in an $n$-universal criterion set by another lattice under certain conditions.
Using this result, we provide the proof of Theorems~\ref{thm:criterionset} and \ref{thm:criterion nonunique}.
Recall that for an $\co$-lattice $L$, we write $nL := L\mathbin{\perp}\dotsb\mathbin{\perp}L$ ($n$ times).

\begin{lemma} \label{lem:criterion common}
Let $\mathcal{S}_K(n)$ be an arbitrary $n$-universal criterion set.
\begin{enumerate}[\textup{(\alph{enumi})}]
 \item Let $J$ be a rank $n$ additively indecomposable $\co$-lattice with a squarefree volume $\mathfrak{v}J$. Then there exists an element in $\mathcal{S}_K(n)$ that is isometric to $J$.
 \item Let $\overline{\mathcal{I}}(n)=\{L_1, \dots, L_s\}$ be a set of complete representatives of indecomposable unimodular $\co$-lattices of rank at most $n$ modulo isometry. For each $i$ with $1\le i\le s$, there is $L \in \mathcal{S}_K(n)$ that represents $\floor{\frac{n}{\rank L_i}}L_i$.
 \item In \textup{(b)}, let $i$ and $j$ satisfy $1\le i < j\le s$. If $\widetilde{L_i} \in \mathcal{S}_K(n)$ represents $\floor{\frac{n}{\rank L_i}}L_i$ and $\widetilde{L_j} \in \mathcal{S}_K(n)$ represents $\floor{\frac{n}{\rank L_j}}L_j$, then $\widetilde{L_i} \ne \widetilde{L_j}$.
\end{enumerate}
\end{lemma}

\begin{proof}
(a) Write $\mathcal{S}_K(n) = \{ J_1, \dots, J_s \}$. Since $J_1 \mathbin{\perp} \dotsb \mathbin{\perp} J_s$ must be $n$-universal, it represents $J$. Since $J$ is additively indecomposable, there exists an index $i$ with $1\le i\le s$ such that there exists a representation $\sigma: J \to J_i$. Since $\sigma(J)$ has a squarefree volume, the inclusion $\sigma(J) \subseteq J_i$ cannot be proper. This proves (a).

(b) If an $\co$-lattice $L$ is indecomposable unimodular, then an $\co$-lattice $\Lambda$ represents $L$ if and only if the unique orthogonal decomposition of $\Lambda$ into indecomposables contains a component isometric to $L$. Hence, if $L$ is indecomposable unimodular, then $\Lambda$ represents $nL$ if and only if the indecomposable decomposition of $\Lambda$ contains $n$ copies of $L$. Fix an index $i$ with $1\le i\le s$. Suppose on the contrary that no $L \in \mathcal{S}_K(n)$ represents $\floor{\frac{n}{\rank L_i}}L_i$. Using indecomposable decompositions, we may write each $L \in \mathcal{S}_K(n)$ as $L = L' \mathbin{\perp} L''$ in a unique way, where $L' \cong t_L L_i$ for some nonnegative integer $t_L$ and $L_i \nrightarrow L''$. Note that our supposition implies $t_L < \floor{\frac{n}{\rank L_i}}$ for all $L \in \mathcal{S}_K(n)$. Define an $\co$-lattice
\[
 \Lambda = \max\{t_L \mid L\in\mathcal{S}_K(n)\}L_i \mathbin{\perp} \bigperp_{L \in \mathcal{S}_K(n)} L''\text.
\]
Clearly, $\Lambda$ represents all lattices in $\mathcal{S}_K(n)$, so it is $n$-universal. However, it does not represent $\floor{\frac{n}{\rank L_i}}L_i$, which is absurd. This proves (b).

(c) Suppose on the contrary that $L \in \mathcal{S}_K(n)$ represents both $\floor{\frac{n}{\rank L_i}}L_i$ and $\floor{\frac{n}{\rank L_j}}L_j$. Since $L_i$ and $L_j$ are indecomposable unimodular and distinct, $L$ represents $L' := \floor{\frac{n}{\rank L_i}}L_i \mathbin{\perp} \floor{\frac{n}{\rank L_j}}L_j$. Since $\rank L_i \le n$, we have $\floor{\frac{n}{\rank L_i}} \rank L_i > \frac{n}2$. Hence, $\rank L' > n$, which is absurd. This proves (c).
\end{proof}

\begin{prop} \label{pro:criterion replacement}
 Let $n$, $j$, $q$, $r$ be positive integers that satisfy $n = qj + r$ and $r \le j-4$.
 Suppose that there exists an indecomposable unimodular $\co$-lattice $J$ of rank $j$.
 Then, for any $n$-universal criterion set $\CS Kn$ and for any $\co$-lattice $R$ of rank $r$, the set
 \[
  (\CS Kn \setminus \{ L \in \CS Kn \mid L \text{ represents } qJ \}) \cup \{ qJ \mathbin{\perp} R \}
 \]
 is also an $n$-universal criterion set.
\end{prop}

\begin{proof}
Denote $\mathcal{S}' := \CS Kn \setminus \{ L \in \CS Kn \mid L \text{ represents } qJ \}$ and let $\Lambda$ be a lattice that represents all lattices in $\mathcal{S}' \cup \{ qJ \mathbin{\perp} R \}$. It suffices to show that $\Lambda$ represents all lattices in $\CS Kn$. Since $\Lambda$ readily represents all lattices in $\mathcal{S}'$, it is enough to show that $\Lambda$ represents all lattices of the form $qJ \mathbin{\perp} R'$ for some rank $r$ lattice $R'$.

Let $R'$ be any $\co$-lattice of rank $r$. Let $\overline{\mathcal{I}}(j-1) = \{ L_1, \dots, L_s \}$ be a set of complete representatives modulo isometry of indecomposable unimodular $\co$-lattices of rank at most $(j-1)$. Clearly $\overline{\mathcal{I}}(j-1)$ does not contain a copy of $J$. Hence by Lemma~\ref{lem:criterion common}, $\Lambda$ must represent the lattice
\[
 qJ \mathbin{\perp} \floor{\frac{n}{\rank L_1}}L_1 \mathbin{\perp} \dotsb \mathbin{\perp} \floor{\frac{n}{\rank L_s}}L_s\text.
\]
Since $\floor{\frac{n}{\rank L_1}}L_1 \mathbin{\perp} \dotsb \mathbin{\perp} \floor{\frac{n}{\rank L_s}}L_s$ represents all lattices in $\gen(I_{j-1})$, it is $(j-4)$-universal, and therefore represents $R'$. Hence, $\Lambda$ represents $qJ \mathbin{\perp} R'$, as required.
\end{proof}

\subsection{Proof of Theorem~\ref{thm:criterionset}}
Let $n \geq 2$ and let $C > 0$ be a fixed constant such that $K$ admits an $n$-universal criterion set $\CS{K}{n}$ which satisfies $\card{ \CS{K}{n} } < C$.

By Lemma~\ref{lem:criterion common}, for each indecomposable unimodular $\co$-lattice $L$ of rank at most $n$, there corresponds a lattice in $\CS{K}{n}$ that represents $\floor{\frac{n}{\rank L}}L$, and they are all distinct. This therefore implies that 
$$\card{ \CS{K}{n} } \ge \IUL{K}{1} + \dotsb + \IUL{K}{n} \ge \max \{ \IUL{K}{1}, \dots, \IUL{K}{n} \}.$$

Thus, by using the lower bounds for $\max \{ \IUL{K}{1}, \dots, \IUL{K}{n} \}$ proven above in Proposition~\ref{prop:maxIULbounds}, this implies that $d$ and $\Delta_K$ satisfy the same bounds as given in (\ref{eq:maindbound}) and (\ref{eq:maindiscbound}) in the case where $n \geq 3$, and satisfies the same bound given in (\ref{eq:maindiscbound2}) in the case $n = 2$.
\qed

\subsection{Proof of Theorem~\ref{thm:criterion nonunique}}
We first prove that there exists a positive integer $N \ge 6$ such that $m\ge N$ implies the existence of an indecomposable unimodular $\co$-lattice whose rank $j$ satisfies the condition $m \le j \le 2m-6$.
Since indecomposable unimodular $\Z$-lattices exist at rank $8$, $12$, and any rank at least $14$ (e.g. see~\cite{OMeara1975}), we may take $N = 7$ when $K = \Q$. By \cite[Corollary~4.2]{Wang}, if a unimodular $\Z$-lattice $L$ is indecomposable, then so is $\co$-lattice $L \otimes \co$ for any real quadratic field $K$. Hence, when $d=2$, then we may still take $N = 7$. When $d \ge 3$, then Theorem~\ref{thm:wrightasymp} guarantees the existence of such an $N$.

Fix an integer $N$ that satisfies the condition in the previous paragraph and let $n \ge 2N-5$.
Then there is an indecomposable unimodular lattice $J$ whose rank $j$ satisfies $\ceil{\frac n2} + 2 \le j \le n-1$.
Let $\CS Kn$ be an $n$-universal criterion set of the smallest cardinality.
If we find infinitely many $n$-universal criterion sets that have the same cardinality as $\CS Kn$, then we are done.

Note that $1 \le n-j \le \floor{\frac n2} - 2 \le j-4$.
Hence, by Proposition~\ref{pro:criterion replacement}, there exists exactly one lattice $L \in \CS Kn$ such that $L \cong J \mathbin{\perp} R$ for a lattice $R$ of rank $n-j$.
Now, by the same proposition, any set of the form $(\CS Kn \setminus \{L\}) \cup \{J \mathbin{\perp} R'\}$ for a lattice $R'$ of rank $n-j$ is also an $n$-universal criterion set of the same cardinality.
This proves the theorem.
\qed

\section{Further problems}

Going beyond the theorems we've proven, there are still a myriad of open problems regarding the behaviour of $\MRU{K}{n}$, both for small and large $n$.  In the hopes of encouraging further research on this topic, we thus pose the following open problems:  

\begin{itemize}
    \item Can one prove Theorem~\ref{thm:mainasymp} for the remaining $\numquadexceptions$ exceptional real quadratic fields $K$?

    \item For any fixed totally real field $K$, can we derive an explicit asymptotic for $\MRU{K}{n}$ itself (and not just $\log \MRU{K}{n}$) as $n \to \infty$?

    \item Can we derive analogous density bounds (in a similar spirit to Theorems~\ref{thm:finite} and \ref{thm:finite2}) for $\MRU{K}{1}$?

    \item Can $\MRU{K}{n+1} - \MRU{K}{n}$ be arbitrarily large for any fixed $n \geq 1$?

    \item Furthermore, given any constant $C > 0$ and $n \geq 1$, are there always only finitely many totally real fields $K$ such that $\MRU{K}{n+1} - \MRU{K}{n} < C$?  If so, can all such fields be effectively determined?

    \item For any fixed totally real field $K$, can we obtain explicit criterion sets $\CS{K}{n}$ for any $\co$-lattice to be $n$-universal?

    \item For any fixed totally real field $K$, can we explicitly compute all $n$ for which there exists a \emph{unique} minimal $n$-universal criterion set $\CS{K}{n}$?
\end{itemize}

\bibliographystyle{alpha}  
\bibliography{bib}

\end{document}